\documentclass[a4paper,10pt]{article}
\usepackage{amssymb}
\usepackage{latexsym}
\usepackage{amsmath}
\usepackage{amsthm}
\usepackage{amsfonts}
\usepackage{amssymb}
\usepackage{graphicx}

\newtheorem{Th}{Theorem}
\newtheorem{Prop}{Proposition}
\newtheorem{Co}{Corollary}
\newtheorem{Lm}{Lemma}

\newtheorem{Rm}{Remark}

\newcommand{\be}{\begin{equation}}
\newcommand{\ee}{\end{equation}}
\newcommand{\R}{\mathbb{R}}
\newcommand{\N}{\mathbb{N}}
\newcommand{\C}{\mathbb{C}}
\newcommand{\F}{\mathcal{F}}
\newcommand{\T}{\mathbb{T}}
\newcommand{\Z}{\mathbb{Z}}

\newcommand\res{\mathop{\hbox{\vrule height 7pt width .5pt depth 0pt
\vrule height .5pt width 6pt depth 0pt}}\nolimits}

\def\lf{\left}
\def\rg{\right}

\def\al{\alpha}
\def\la{\lambda}
\def\e{\varepsilon}

\def\ds{\displaystyle}
\def\ov{\overline}
\def\Om{\Omega}
\def\om{\omega}
\def\p{\partial}
\def\bn{\vec{n}}

\def\bbe{\vec{e}}
\def\bbf{\vec{f}}
\def\bH{\vec{H}}
\def\bu{\vec{u}}
\def\bv{\vec{v}}

\def\bC{\vec{C}}
\def\bA{\vec{A}}
\def\bB{\vec{B}}
\def\bD{\vec{D}}
\def\bE{\vec{E}}
\def\bF{\vec{F}}
\def\bL{\vec{L}}
\def\bR{\vec{R}}

\def\bw{\vec{w}}

\def\bP{\vec{\Phi}}
\def\bPsi{\vec{\Psi}}

\def\res{\mathop{\hbox{\vrule height 7pt width .5pt 
depth 0pt\vrule height .5pt width 6pt depth 0pt}}\nolimits}
\newcommand{\D}{D}

\newcommand{\Sp}{\mathbb{S}}

\renewcommand{\div}{{\rm{div}}}
\def\bI{\mathbb I}
\def\bbI{\vec{\mathbb I}}

\def\k{{\rm k}}

\begin{document}
\title{A frame energy for immersed tori and applications to regular homotopy classes}
\author{Andrea Mondino\footnote{Department of Mathematics, ETH Zentrum,
CH-8093 Z\"urich, Switzerland, @mail: andrea.mondino@math.ethz.ch. \newline A. M. is supported by the ETH fellowship.} , Tristan Rivi\`ere\footnote{Department of Mathematics, ETH Zentrum,
CH-8093 Z\"urich, Switzerland.}}
\date{ }
\maketitle

{\bf Abstract :}{\it The paper is devoted to study the Dirichelet energy of moving frames on 2-dimensional tori immersed in the euclidean $3\leq m$-dimensional space.  This functional, called Frame energy, is naturally linked to the Willmore energy of the immersion and on the conformal structure of the abstract underlying surface. As first result, a Willmore-conjecture type lower bound is established : namely for every torus immersed in $\R^m$, $m\geq 3$, and any moving frame  on it, the frame energy is at least $2\pi^2$ and equalty holds if and only if $m\geq 4$, the immersion is the standard Clifford torus  (up to rotations and dilations), and the frame is the flat one. Smootheness of the critical  points of the frame energy is proved after the discovery of hidden conservation laws and, as application, the minimization of the Frame energy in regular homotopy classes of immersed tori in $\R^3$ is performed.
}
\medskip

\noindent{\bf Math. Class.} 30C70, 58E15, 58E30, 49Q10, 53A30, 35R01, 35J35, 35J48, 35J50.

\section{Introduction}\label{Sec:Intro}
The purpose of this paper is to  study the Dirichlet energy of moving frames associated to  tori immersed in $\R^m$, $m\geq 3$. Moving frames have been played a key role  in the modern theory of immersed surfaces starting from the pioneering works of Darboux \cite{Darb},  Goursat \cite{Gours} ,  Cartan \cite{Cart}, Chern \cite{Chern1}-\cite{Chern2}, etc. (note also that in the book of Willmore \cite{Will}, the theory of surfaces is presented from Cartan's point of view of moving frames, and the recent book of  H\'elein \cite{Hel} is devoted to the role of moving frames in modern analysis of submanifolds; see also the recent introductory  book  of Ivey and Landsberg \cite{IvLan}).  Indeed, due to the strong link between moving frames on an \emph{immersed} surface and the conformal structure of the underlying \emph{abstract} surface (see later in the introduction for more explanations), the importance of selecting  a ``best moving frame'' in surface theory is  comparable to fixing an optimal gauge in physical problems (for instance for the study of Einstein's equations of general relativity it is natural  to work in the gauge of the so called harmonic coordinates, for the analysis of Yang-Mills equation it is convenient  the so called Coulomb gauge, etc.).
\\

Before going to the description of the main  results of the present paper, the  objects of the investigation of the this work must be defined.
\\
Let $\T^2$ be the abstract 2-torus  (seen as 2-dimensional smooth manifold) and let $\bP:\T^2\hookrightarrow \R^m, m\geq 3,$ be a smooth immersion (let us start with smooth immersions, then we will move to weak immersions). One denotes with $T\bP(\T^2)$ the tangent bundle to $\bP(\T^2)$, a pair $\bbe:=(\bbe_1,\bbe_2)\in \Gamma(T\bP(\T^2))\times \Gamma(T\bP(\T^2))$ is said \emph{a moving frame} on $\bP$ if, for every $x \in \T^2$, the couple $(\bbe_1(x),\bbe_2(x))$ is a  positive orthonormal basis for $T_x\bP(\T^2)$ (with positive we mean that we fix a priori an orientation of $\bP(\T^2)$ and that the moving frame agrees with it). 

Given $\bP$ and $\bbe$ as above we define the \emph{frame energy} as the Dirichelet energy of the frame, i.e.
\be\label{eq:defFrameEn}
\F(\bP,\bbe):=\frac{1}{4}\int_{\T^2} |d \bbe|^2 \, dvol_g,
\ee
where $d$ is the exterior differential along $\bP$, $dvol_g$ is the area form given by the immersion $\bP$ (this can be seen equivalently as the restriction to $\bP(\T^2)$ of the 2-dimensional Hausdorff measure on $\R^m$, or as the volume form associated to the pullback metric $g:=\bP^*(g_{\R^m})$ where $g_{\R^m}$ is the euclidean metric on $\R^m$), and $|d \bbe|$ is the length of the exterior differential of the frame which is given in local coordinates by $|d \bbe|^2=\sum_{k=1}^2 |d \bbe_k|^2=\sum_{i,j,k=1}^2  g^{ij} \p_{x_i} \bbe_{k} \cdot \p_{x_j} \bbe_k$; in the paper $\bu\cdot \bv$ or $(\bu,\bv)$ denotes the scalar product of vectors in $\R^m$.

Let $\bn$ be the unit simple $m-2$-multivector giving   the  normal space to $\bP$ in terms of the Hodge duality operator in $\R^m$
\be\label{eq:defbn}
\bn:=\star_{\R^m} \frac{\partial_{x_1}\bP \wedge \partial_{x_2} \bP}{|\partial_{x_1}\bP \wedge \partial_{x_2}, \bP|}\quad.
\ee
In $\R^3$, for instance, it can be written in terms  of vector product
\be\label{eq:defbnR3}
\bn:= \frac{\partial_{x_1}\bP \times \partial_{x_2} \bP}{|\partial_{x_1}\bP \times \partial_{x_2} \bP|} \quad.
\ee
Let $\pi_T:\R^m\to T\bP(\T^2)$ and   $\pi_{\bn}:\R^m\to N\bP(\T^2)$ the orthonormal projections on the tangent and on the normal space  respectively. Recall that the second fundamental form $\bbI$ of the immersion $\bP$ is defined by
\be\label{eq:defbbI}
\bbI_{ij}:=\pi_{\bn} (\partial^2_{x_i x_j} \bP)
\ee
and the mean curvature $\bH$ is given by half of its trace
\be\label{eq:defbH}
\bH:=\frac{1}{2} g^{ij} \, \bbI_{ij}\quad .
\ee 
\\Notice that, writing $d\bbe_i=\pi_T(d \bbe_i)+\pi_{\bn}(d \bbe_i)=(d\bbe_i, \bbe_{i+1}) \bbe_{i+1}+\pi_{\bn}(d\bbe_i)$-where $\Z_2$-indeces are used-, the frame energy decomposes as
\be\label{eq:F=FT+bbI}
\F(\bP, \bbe)= \frac{1}{2} \int_{\T^2} |\bbe_1\cdot d \bbe_2|_g^2 \, dvol_g + \frac{1}{4}\int_{\T^2} |\bbI|^2 \, dvol_g\quad . 
\ee
The tangential part
\be\label{def:FT}
\F_T(\bP, \bbe):= \frac{1}{2} \int_{\T^2} |\bbe_1\cdot d \bbe_2|_g^2\, dvol_g,
\ee
is equal to the $L^2$-norm of the covariant derivative of the frame with respect to the Levi-Civita connection, whereas the normal part corresponds to the Willmore energy after having applied Gauss Bonnet theorem: \be\label{eq:WHbbI}
W(\bP):=\int_{\T^2} |\bH|^2 \, dvol_g= \frac{1}{4}\int_{\T^2} |\bbI|^2 \, dvol_g, 
\ee
where the above defined $W$ is the so-called Willmore functional, we get
 \be\label{eq:FFTW}
\F(\bP, \bbe)=\F_T(\bP, \bbe)+ W(\bP)\quad.
\ee

Let us observe that the frame energy $\F$ is invariant under scaling and under conformal transformations of the metric $g$, but not under conformal transformations of $\R^m$ (to this purpose note that, by definition, the moving frame has to be orthonormal with respect to the \emph{extrinsic} metric, i.e. $g_{\R^m}$, but the norm of the derivative as well as the volume form is computed with respect to the \emph{intrinsic} metric $g$). Therefore, even if natural on its own, $\F$ can be seen as a more coercive Willmore energy where the extra term $\F_T$ prevents the degenerations caused by the action of the Moebius group of $\R^m$ and  the degeneration of the confomal class of  the abstract torus. More precisely we have the  the following proposition.

\begin{Prop}\label{Prop:BoundCC}
For every $C>0$, the metrics induced by  the framed immersions in $\F^{-1}([0,C])$ are contained in a compact subset of the moduli space of the torus.
\end{Prop}



Let us mention that the proof of Proposition \ref{Prop:BoundCC} is remarkably elementary and makes use just of the Fenchel lower bound \cite{Fenchel} on the total curvature of a closed curve in $\R^m$.  
\\Combining Proposition  \ref{Prop:BoundCC} with the celebrated results of Li-Yau \cite{LY} and  Montiel-Ros \cite{MonRos} on the Willmore conjecture, we manage to prove the following sharp lower bound (with rigidity) on the frame energy.

\begin{Th}\label{thm:LB}
Let $\bP:\T^2\hookrightarrow \R^m$ be a smooth immersion of the 2-dimensional torus into the Euclidean $3\leq m$-dimensional space and let $\bbe=(\bbe_1,\bbe_2)$ be any moving frame along $\bP$. 

Then the  following lower bound holds:
\be\label{eq:LB}
\F(\bP,\bbe):=\frac{1}{4}\int_{\T^2} \lf|d\bbe\rg|^2 \, dvol_g \geq 2 \pi^2\quad.
\ee  
Moreover, if in \eqref{eq:LB} equality holds then it must be $m\geq 4 $, $\bP(\T^2)\subset \R^m$ must be, up to isometries and dilations in $\R^m$, the Clifford torus 
\be\label{eq:defTCl}
T_{Cl}:=S^1\times S^1 \subset \R^4\subset \R^m\quad,
\ee
and $\bbe$ must be, up to a constant rotation on $T(\bP(\T^2))$, the moving frame given by $(\frac{\p}{\p \theta}, \frac{\p}{\p\varphi})$, where of course $(\theta,\varphi)$ are natural flat the coordinates on $S^1\times S^1$. 
\end{Th}

\begin{Rm}\label{rem:LB}
Let us mention that, thanks to  \eqref{eq:FFTW},  in \emph{codimension one},  the lower bound  \eqref{eq:LB} follows by the recent proof of the Willmore conjecture by Marques and Neves \cite{MN} using min-max principle; the approach here is a more direct energy based consideratton. Indeed from their result non just the frame energy, but the Willmore functional $W(\bP)$ is bounded below by $2\pi^2$ for any smooth immersed torus, and $W(\bP)=2\pi^2$ if and only if $\bP$ is a conformal transformation of the Clifford torus.  
Curiously, our lower bound seems to work better in codimension at least two, where it becomes sharp and rigid; clearly, in codimension  one it is not sharp  because of the nonexistence of flat immersions of the torus in $\R^3$ 
and because of the Marques-Neves proof of the Willmore conjecture. 
\\Let us also mention that Topping  \cite[Theorem 6]{Topp}, using  arguments  of integral geometry (very far from our proof), obtained an analogous lower bound on an analogous frame energy for immersed tori in $S^3$ under the assumption that the underlying conformal class of the immersion is a \emph{rectangular}  flat torus. 
\end{Rm}

For variational matters the framework of smooth immersions has to be relaxed to a weaker notion of immersion introduced by the second author in \cite{Riv2}, that we recall below.
\\

Given any smooth reference metric $g_0$ on $\T^2$ (the definition below is independent of the choice of a smooth $g_0$), the map $\bP:\T^2\to \R^m$ is called \emph{weak  immersion} if the following properties hold
\begin{enumerate}
\item $\bP\in W^{1,\infty}(\T^2, \R^m)$ and the pullback metric $g_{\bP}:=\bP^*g_{\R^m}$ is equivalent to $g_0$, i.e. there exists a constant $C_{\bP}>1$ such that
$$C_{\bP}^{-1} \, g_{\bP}\leq g_0 \leq C_{\bP} \,  g_{\bP} \quad \text{as quadratic forms} \quad,$$
\item denoted by $\bn\in L^\infty(\T^2,\Lambda^{m-2}\R^m)$ the normal space defined a.e. by \eqref{eq:defbn} (or more simply in $\R^3$ by \eqref{eq:defbnR3}), it holds $\bn \in W^{1,2}(\T^2)$; or, equivalently, the second fundamental form $\bbI$ defined a.e. in \eqref{eq:defbbI} is $L^2$ integrable over $\T^2$. 
\end{enumerate}
The space of weak immersions $\bP$ from $\T^2$ into $\R^m$ is denoted by ${\cal E}(\T^2, \R^m)$. 
Recall also that given a weak immersion, up to a local bilipschitz diffeomorphism, we can assume it is locally conformal so it induces a smooth conformal structure on the torus (this result is a consequence of a  combination of works of Toro \cite{To1}-\cite{To2}, M\"uller-Sverak \cite{MS}, H\'elein \cite{Hel} and the second author \cite{Riv2}; for a comprehensive discussion see   \cite{RiCours}). 
\\

Let us remark that Proposition \ref{Prop:BoundCC} and Theorem \ref{thm:LB} holds for weak immersions as well (for more details see Section \ref{Sec:RegHom}).  
\\

In order to perform the calculus of variations of the frame energy, in Section \ref{SubSec:FirstVar} we establish that the Frame energy is differentiable in ${\cal E}(\T^2, \R^m)$ and we compute the first variation of the tangential frame energy $\F_T$ which, combined with the first variation of the Willmore functional \cite{Riv1} and with \eqref{eq:FFTW}, gives the first variation of the frame energy $\F$.
As for the Willmore energy (as well as for many important geometric problems as Harmonic maps, CMC surfaces, Yang Mills, Yamabe, etc.) the equation we obtain is \emph{critical}. It is therefore challenging to prove the regularity of critical points of the frame energy.
\\

 Inspired by the work of H\'elein  \cite{Hel} on CMC surfaces and of the second author on Willmore surfaces \cite{Riv1} (see also \cite{MoRi2} for the manifold case and \cite{RiCours} for a comprehensive discussion), in order to study the regularity of the critical points of the frame energy we discover some new hidden conservation laws: in Subsection \ref{SubSec:GenConsLaw} we find some new identities   for general weak conformal immersions, and then in Subsection \ref{SubSec:SysConsLaw} we use these identities in order  to deduce a system of conservation laws satisfied by the critical points of the frame energy. In particular, this system of conservation laws  yields to an elliptic system involving Jacobian nonlinearities which can be studied using \emph{integrability by compensation theory} (for a comprehensive treatment see \cite{RivIntComp}).
 Thanks to this special form, we are able to show regularity of the solutions of this critical system, namely we  prove the following result.

\begin{Th}\label{thm:regularity}
Let $\bP$ be a weak immersion of the disc $\D^2$ into $\R^3$ and let $\bbe=(\bbe_1,\bbe_2)$ be a moving frame on $\bP$ such that $(\bP,\bbe)$ is a critical point of the frame energy $\F$. Then, up to a bilipschitz reparametrization we have locally that $\bP$ is conformal and  $\bbe$ is the coordinate moving frame associated to $\bP$, i.e. $(\bbe_1,\bbe_2)=\lf(\frac{\p_{x_1} \bP }{|\p_{x_1}\bP|}, \frac{\p_{x_2} \bP }{|\p_{x_2} \bP|} \rg)$. Moreover, there exist $\rho\in (0,1)$ such that $\bP|_{B_\rho(0)}$ is a $C^\infty$ immersion. 
\end{Th}
Let us observe that for sake of simplicity of the presentation, this work is more focused in the codimension one case; the higher codimensional case will be the object of a forthcoming paper, many of the  arguments carry on in a similar way. 
\\

Now let us discuss an application of the tools developed in this paper to the study of regular homotopy classes of immersions.
\\

Let us first recall some classical facts about regular homotopies,  starting from the definition: given a smooth closed surface $\Sigma^2$, two smooth  immersions $f,g:\Sigma^2 \hookrightarrow \R^m$ are said \emph{regularly homotopic} if there exists a smooth map  $H:\Sigma^2 \times [0,1] \to \R^m$, called \emph{regular homotopy between $f$ and $g$}, such that $H(\cdot,0)=f(\cdot)$, $H(\cdot,1)=g(\cdot)$ and $H_t(\cdot):=H(\cdot,t): \Sigma^2 \hookrightarrow \R^m$ is an immersion for every $t\in[0,1]$, everything up to diffeomorphisms of $\Sigma^2$.
\\

In his  celebrated paper \cite{Smale58} of 1958, Smale proved that any couple of smooth immersions 
of the 2-sphere into $\R^3$ are regularly homotopic, i.e. homotopic via a one parameter family of immersions (see also \cite{Smale59} for the higher dimensional results). The same is not true for immersions of the 2-sphere in $\R^4$ where indeed there are countably many regular homotopy classes. An year later, Hirsch \cite{Hirsch59} generalized the ideas of Smale to arbitrary submanifolds and in particular he proved that the regular homotopy classes of immersions of any fixed smooth closed surface in a Euclidean space of codimension higher  than two trivialize, i.e. every two immersions of a fixed surface are regularly homotopic (this follows from the fact that the second homotopy group of the Stiefel manifold  $V_2(\R^m)$ is null for $m\geq 5$ ). 
\newline

Remarkably, the case of  tori immersed in $\R^3$ differs from the one of the spheres. Indeed, as proved by  Pinkall in 1985 \cite{Pink85}, there are exactly two regular homotopy classes of immersed tori in $\R^3$: the standard one (the one of a classical rotational torus, say ) and the  nonstandard one (a knotted torus, for an explicit example we refer to   \cite{Pink85}). One could address the question of a canonical rapresentant for each of the two classes.
\\

 As an application of the tools developed in this paper, we prove the existence of a smooth  minimizer of the frame energy within each of the two regular homotopy classes; such a minimizer can be seen as a canonical raprensentant of its  regular homotopy class.  It is proven below   that  the notion of regular homotopy class extend to the general framework of weak immersions (see Proposition \ref{prop:sigmaWeak}).  The following is the last   main result of the present paper.




\begin{Th}\label{thm:MinRegHom}
Fix $\sigma$ a regular homotopy class of  immersions of the 2-torus $\T^2$ into $\R^3$. 
Then there exists a smooth conformal immersion $\bP:\T^2\hookrightarrow \R^3$, with $\bP\in \sigma$, such that, called $\bbe:=(\bbe_1, \bbe_2):=\lf(\frac{\p_{x_1}\bP}{|\p_{x_1}\bP|},\frac{\p_{x_2}\bP}{|\p_{x_2}\bP|} \rg)$ the coordinate moving frame, the couple $(\bP,\bbe)$ minimizes the frame energy $\F$ among all weak immersions of $\T^2$ into $\R^3$ lying in $\sigma$ and all $W^{1,2}$ moving frames on $\bP(\T^2)$:
\be\label{eq:MinRegHom}
\F ( \bP ,\bbe)=\min \lf \{\F(\tilde{\bP}, \tilde{\bbe}): \tilde{\bP}\in {\cal E}(\T^2,\R^3), \tilde{\bP}\in \sigma, \,\tilde{\bbe}\in W^{1,2}(\T^2)
\rg \}.
\ee
\end{Th} 

Let us conclude the introduction with some comment and open problems.
As already observed in Remark \ref{rem:LB}, from Theorem \ref{thm:LB} it follows that the global minimizer of the frame energy for (weak) immersions of $\T^2$ into $\R^m$, for $m\geq 4$, is the Clifford torus; instead it is still an open problem to identify who  is the minimizer for immersions into $\R^3$. We expect it to be the Clifford torus as well. We also expect the minimizer of the nonstandard regular homotopy class of  immersed tori into $\R^3$ to be the diagonal double cover of the Clifford torus proposed by Kusner in the framework of the Willmore  problem \cite[page 333]{Kus}. Both of these are open problems, as well as the existence of a minimizer of the frame energy among regular homotopy classes of tori immersed into $\R^4$; indeed, in codimension two, there is  an extra difficulty given by the possibility of having loss of homotopic complexity in the concentration points of the frame energy (to exclude this, in our argument we use Lemma \ref{Lm:DiskHomotopy}, which is not true in codimension two). Notice finally that in codimension greater or equal to three, by the aforementioned result of Hirsch, there is just one regular  homotopy class of immersed tori, and by Theorem \ref{thm:LB} the global minimizer is the Clifford torus, up to isometries and rescalings, with rigidity.  
\\

The paper is organized as follows. 
\\Section \ref{Sec:LB} is devoted to the proofs of Proposition \ref{Prop:BoundCC} and Theorem \ref{thm:LB}, namely  the bound on the conformal class and the  lower bound  (with rigidity) on the frame energy.

In Section  \ref{Sec:ConsLaw} it is established the system of conservation laws satisfied by the critical points of the Frame energy; more precisely, in Subsection \ref{SubSec:FirstVar} we establish the Frechet differentiability of $F$ in the space of weak immersion and compute the first variation formula, in Subsection \ref{SubSec:GenConsLaw} we discover some general conservation laws associated to a general weak conformal immersion, and in Subsection \ref{SubSec:SysConsLaw}  these conservation  laws are used  to obtain a system  of conservation laws involving jacobian quadratic non linearities satisfied by the critical points of the Frame energy.

In  Section \ref{Sec:Reg} the peculiar form of the aforementioned system is exploited in order to deduce the regularity of the critical points of the Frame energy via the theory of integrability by compensation, namely Theorem \ref{thm:regularity} is proved.

In Section  \ref{Sec:RegHom} the above tools of the calculus of variations are applied to prove the existence of a minimizer of the Frame energy in regular homotopy classes, namely Theorem \ref{thm:MinRegHom}.

Finally in the Appendices we recall some classical geometric computations in conformal coordinates used in Section \ref{Sec:ConsLaw}, a Lemma of functional analysis used in the proof of the regularity theorem, and a Lemma of differential topology used in the proof of Theorem \ref{thm:MinRegHom}.

\section{A lower bound -with rigidity- for the frame energy in $\R^m$, the analogue of the Willmore conjecture}\label{Sec:LB}

\subsection{Reduction to conformal immersions of flat tori and coordinate moving frames}\label{subsec:LBRed}
Let $\bP:\T^2 \hookrightarrow \R^m$ be a smooth immersion of the torus into the euclidean $3\leq m$-dimensional space. 


The goal of this section is to prove Lemma \ref{Lm:RedctLB}: namely to  reduce the problem of calculating the infimum of the frame energy among all smooth immersions of $\T^2$ into $\R^m$ and all moving frames
, to the case of coordinate moving frames associated to smooth conformal immersions of tori lying in the moduli space of conformal structures. We will proceed with consecutive reductions.
\newline

\emph{Reduction 1: $\bbe$ satisfies the Coulomb condition.} Since in this section we are interested in giving a lower bound on the frame energy $\F$, we can assume that the frame $\bbe$ minimizes its tangential part $\F_T:=\frac{1}{2}\int_{\T^2} |\bbe_2, \cdot d \bbe_1|_g^2 \, dvol_g$; this is equivalent to say that $\bbe$ is a \emph{Coulomb} frame, i.e. it satisfies the Coulomb condition 
\be\label{eq:Coulomb}
d^{*_g}(\bbe_1,d \bbe_2)=0,
\ee
which reads in local isothermal coordinates as $\div(\bbe_1,\nabla \bbe_2)=0$ (for more details about Coulomb frames  and the Chern method see \cite{Hel} or \cite{RiCours}).
\newline

\emph{Reduction 2: $\bbe$ is a coordinate moving frame.} Recall that, by using the Chern moving frame method and the fact that $\bbe$ is Coulomb,  we can can cover the torus $\T^2$ by finitely many balls $\{B_k\}_{k=1,\ldots,N}$ such that for every ball  there exists a diffeomorphism $f_k:B_k\to B_k$ such that $\bP\circ f_k$ is a smooth  conformal immersion of $B_k$ into $\R^m$ and 
\be\label{eq:ejpPhij}
\bbe_j=\frac{\p_{x_j} (\bP\circ f_k)}{|\p_{x_j} (\bP\circ f_k)|},
\ee 
i.e. the moving frame $\bbe$ is the coordinate moving frame associated to the smooth conformal immersion $\bP$.

\emph{Reduction 3: the reference torus is flat.} Now the local conformal coordinates on $\T^2$ define a smooth conformal structure on $\T^2$ therefore, by the Uniformization Theorem, there exists a diffeomorphism $\psi$ from a flat torus $\Sigma$ (i.e. $\Sigma$ is the quotient of $\R^2$ modulo a $\Z^2$ lattice) into our $\T^2$ such that $f_k^{-1}\circ \psi$ is a conformal diffeormorphism; it follows that $\bP\circ \psi=\bP\circ f_k\circ f_k^{-1}\circ \psi$ is a smooth conformal immersion of $\Sigma$ into $\R^m$. Moreover, recalling that the property of being Coulomb for a moving frame is invariant under conformal changes of metric (this property is a direct consequence of equation \eqref{eq:Coulomb} and of the invariance of the Hodge operator $*_g$ under conformal changes of metric), we get that $\bbe\circ \psi$ is a Coulomb moving frame on $\Sigma$. 

Now observe that we have another natural Coulomb frame on the flat torus $\Sigma$  given by the conformal immersion $\bP\circ \psi$, namely
\be\label{eq:bbf}
\bbf_j=\frac{\p_{y_j} (\bP\circ \psi)}{|\p_{y_j} (\bP\circ \psi)|}, \quad \text{satisfying  } \;  \bbf_1\cdot d \bbf_2=*d \la, 
\ee
where $y_1,y_2$ are standard coordinates on $\Sigma$ and $\la=\log(|\p_{y_i} (\bP\circ \psi)|)$ is the conformal factor. Clearly, we can the two moving frames $\bbe$ and $\bbf$ via a rotation in the tangent space, i.e.
$$\bbe_1+i\bbe_2=e^{i\theta} (\bbf_1+i \bbf_2) $$
for some smooth function $\theta:\Sigma\to S^1$. Observe that, using  \eqref{eq:bbf} and integrating by parts we get
\begin{eqnarray}
\F_T(\bP,\bbe)&=&\int_{\Sigma} |\bbe_1 \cdot d\bbe_2|^2 dy=\int_{\Sigma} |\bbf_1\cdot d\bbf_2-d\theta|^2  dy=\int_{\Sigma} |\bbf_1\cdot d\bbf_2|^2 + |d\theta|^2- 2<\bbf_1 \cdot d\bbf_2, d \theta>  dy\nonumber \\
&=& \int_{\Sigma} |\bbf_1 \cdot d \bbf_2|^2 + |d\theta|^2-2 < * d \lambda, d \theta>  dy =  \int_{\Sigma} |\bbf_1\cdot d\bbf_2|^2 + |d\theta|^2 dy \nonumber \\
&\geq&  \F_T(\bP,\bbf)  \quad, \label{est:bbebbf}
\end{eqnarray}
with equality if and only of $\bbe$ is a constant rotation of $\bbf$ (this will be useful to prove the part of the rigidity statement involving the frame). 
\newline

\emph{Reduction 4: the flat torus lies in the moduli space of conformal structures.} As last reduction, we want to reduce the problem to the case when  $\Sigma$ is a flat torus in the canonical moduli space $\cal M$ of conformal structures of tori composed by the parallelograms in $\R^2$ whose edges are $(1,0)$ and $\tau=(\tau_1,\tau_2)\in M$  where $M$ is the strip
\be\label{eq:defM}
M:=\{\tau=(\tau_1,\tau_2)\in \R^2:\tau_2>0, \; -\frac{1}{2}<\tau_1 <\frac 12, \;  |\tau|\geq 1 \text{ and } \tau_1\geq 0 \text{ if } |\tau|=1 \}\quad.
\ee
Indeed a classical result of Riemann surfaces (see for instance \cite[Section 2.7]{Jo}) says that up to composition with a linear transformation which preserves the orientation (more precisely up to composition with a  projective unimodular trasformation in $PSL(2,\Z)$), the conformal structure of our flat torus $\Sigma$ is isomorphic to the one of a flat torus described by the parallelogram given by $(1,0)$ and $(\tau_1,\tau_2)\in M$, where $M$ was defined in \eqref{eq:defM}.   
We can finally summarize the discussion in the following lemma. 

\begin{Lm}\label{Lm:RedctLB}
Let $\T^2$ be the abstract torus (i.e. the unique  smooth orientable $2$-dimensional  manifold of genus one).

Call $\beta^m_1\geq 0$ the infimum of the frame energy $\F(\bP,\bbe)$ among all smooth immersions $\bP$ of $\T^2$ into $\R^m$ and all the moving frames $\bbe=(\bbe_1,\bbe_2)$  along $\bP$. 

Denote also $\beta^m_2\geq 0$ the infimum of the frame energy $\F(\bP,\bbf)$ among all smooth conformal immersions $\bP$ of any flat torus $\Sigma$ described by any lattice in $\R^2$ of the form $((1,0),(\tau_1,\tau_2))$-where $(\tau_1,\tau_2)\in M$ is defined in \eqref{eq:defM}; here $\bbf$ is the coordinate moving frame associated to $\bP$, i.e. $\bbf_j:=\p_{x_j}\bP/|\p_{x_j}\bP|$, $j=1,2$, $(x_1,x_2)$ being flat coordinates on $\Sigma$.  

Then   $\beta^m_1=\beta^m_2$. In other words, in order to compute the infimum of $\F$ among all smooth immersions of tori and all moving frames
, it is enough to restrict to coordinate moving frames associated to smooth conformal immersions of flat tori lying in $\cal M$, the moduli space of conformal structures of tori.
\end{Lm}  

\subsection{Proof of Theorem \ref{thm:LB}:  lower bound and  rigidity for the frame energy}
From now on we will work with a torus as in the reduced case: $\Sigma=\R^2/(\Z\times \tau\Z)$ is the flat quotient of $\R^2$ by the lattice generated by the two vectors $(1,0),(\tau_1,\tau_2)$-where $(\tau_1,\tau_2)\in M$ is defined in \eqref{eq:defM}; we will denote 
\be\label{eq:deftheta}
\theta_{\Sigma}:= \arccos \tau_1 \in \lf(\frac{\pi}{3},\frac{2\pi}{3} \rg)\quad,
\ee    
where the interval $\lf(\frac{\pi}{3},\frac{2\pi}{3} \rg)$ comes directly from the definition of $M$ as in \eqref{eq:defM}. 

One of the key technical  results of this paper is the control of the conformal class in terms of the frame energy, namely Proposition \ref{Prop:BoundCC}; this is implied by the following lower bound.

\begin{Prop}\label{Prop:LowerBound1}
Let $\Sigma=\R^2/(\Z\times \tau\Z)$, with $\tau \in M$ as above, be a flat torus. Let $\bP:\Sigma \hookrightarrow \R^m$, $m\geq 3$, be a smooth conformal immersion and let $\bbe$ be the coordinate frame associated to $\bP$: $\bbe_j:=\p_{x_j}\bP/|\p_{x_j}\bP|$, $j=1,2$, $(x_1,x_2)$ being flat coordinates on $\Sigma$.  

Then the following lower bound holds true
\be\label{eq:LB0}
\begin{array}{l}
\ds\int_{\Sigma} e^{-4\la}\lf[ \lf(1+\frac{\cos^4 \theta}{\sin^2 \theta}\rg) \, \bbI^2_{11} + \sin^2 \theta \, \bbI^2_{22} + 4 \cos^2 \theta \, \bbI^2_{12} \rg] +\lf[ \lf(1 + \cot^2 \theta  \rg) (d_{\bbe_1}\bbe_1, \bbe_2 )^2+  (d_{\bbe_2}  \bbe_2, \bbe_1)^2  \rg] \, dvol_g\\[5mm]
\ds \qquad \qquad \qquad \qquad\qquad \qquad \qquad \qquad \geq 4\pi^2 \left(\tau_2+\frac{1}{\tau_2} \rg)\,,
\end{array}
\ee
where $\theta:=\theta_\Sigma:= \arccos \tau_1 \in \lf(\frac{\pi}{3},\frac{2\pi}{3} \rg)$, $dvol_g$ is the area form on $\Sigma$ induced by the pullback metric $g=\bP^* g_{\R^m}$, and $\la=\log(|\p_{x_i}\bP|)$ is the conformal factor. In particular \eqref{eq:LB0} implies the following lower bound on the frame energy of $(\bP,\bbe)$:
\be\label{eq:LB1}
\F(\bP,\bbe):= \frac{1}{4} \int_\Sigma \lf|d\bbe \rg|^2\, dvol_g \geq \pi^2 \lf(\tau_2+\frac{1}{\tau_2} \rg) \lf( \frac{\sin^2 \theta}{\sin^2 \theta + \cos^4 \theta} \rg) \quad.
\ee
\end{Prop}

\textbf{Proof of Proposition \ref{Prop:LowerBound1}}
\\
First of all recall that by the classical Fenchel Theorem (the original proof of Fenchel \cite{Fenchel}, see also \cite{DoCcs}, was for closed curves immersed in $\R^3$. The result  was generalized to immersions in $\R^m$, $m\geq 3$, by Borsuk  \cite{Borsuk} with a different proof), given a smooth closed curve $\vec{\gamma}:S^1 \hookrightarrow \R^m$ one has
\be\label{eq:Fenchel}
\int_{\vec{\gamma}} k \, ds \geq 2\pi,
\ee
where $k:=|\frac{d^2}{ds^2} \vec{\gamma}(s)|$ is curvature of $\vec{\gamma}$-here $s$ is the arclenght parameter. The strategy is to apply Fenchel Theorem to the curves $\vec{\gamma}_x:=\bP(\gamma_x(\cdot)),\vec{\gamma}_y:=\bP(\gamma_y(\cdot))$ where $\gamma_x(\cdot):[0,\tau_2]\to \Sigma$ and  $\gamma_y(\cdot):[0,1]\to \Sigma$ are given by
\be\label{eq:defgamma}
\gamma_x(t):=\lf(x+t \cot \theta, t\rg),\quad \gamma_y(t):=\lf(y \cot  \theta+ t, y \rg) \quad,
\ee 
for every $x\in [0,1]$ and $y \in [0,\tau_2]$. Notice that $\gamma_x$ and $\gamma_y$ are nothing but the parallel curves of the vectors generating the lattice  of $\Sigma$.
Now applying Fenchel theorem to $\vec{\gamma}_y$, recalling that $\bP$ is conformal with $\lambda:=\log|\p_{x}\bP|=\log|\p_{y}\bP|$ so that $\p_x \bP=e^\la \bbe_1$, we have

$$
2 \pi \leq \int_{\vec{\gamma}_y} k\, ds =\int_{0}^{L(\vec{\gamma}_y)} \lf|\frac{d}{ds} \dot{\gamma}_y \rg| \, ds= \int_{0}^{L(\vec{\gamma_y})} \lf|d_{\bbe_1} \bbe_1 \rg|\, ds = \int_0^1 \lf|d_{\bbe_1} \bbe_1 \rg| \, e^{\lambda} dx \quad, 
$$
where $L(\vec{\gamma}_y)$ is of course the length of the curve $\vec{\gamma}_y(\cdot)$. Squaring the above inequality, using Cauchy-Schartz and integrating with respect to $y \in [0,\tau_2]$ gives
\be\label{eq:LB00}
4 \pi^2 \tau_2\leq\int_0^{\tau_2} \int_0^1 \lf|d_{\bbe_1} \bbe_1 \rg|^2 \, e^{2\lambda} dx dy= \int_{\Sigma}  \lf|d_{\bbe_1} \bbe_1 \rg|^2 \, dvol_g \quad.
\ee 
Analogously, called $\bbe_2^{\theta}:=(\cos \theta \, \bbe_1, \sin \theta \,  \bbe_2)$, observing that $\dot{\gamma}_x/|\dot{\gamma}_x|=\bbe^2_\theta$ we get
$$
2 \pi \leq \int_{\vec{\gamma}_x} k\, ds =\int_{0}^{L(\vec{\gamma}_x)} \lf|\frac{d}{ds} \dot{\gamma}_x \rg| \, ds= \int_{0}^{L(\vec{\gamma_x})} \lf|d_{\bbe_2^{\theta}} \bbe_2^{\theta} \rg| \, ds = \frac{1}{\sin \theta} \int_0^{\tau_2} \lf|d_{\bbe_2^{\theta}} \bbe_2^{\theta} \rg| \, e^{\lambda} dy \quad. 
$$
Again, squaring the above inequality, using Cauchy-Schartz and integrating with respect to $x \in [0,1]$ gives
\be\label{eq:LB01}
 \frac{4 \pi^2} {\tau_2}\leq\frac{1}{\sin^2 \theta}\int_0^1 \int_0^{\tau_2}  \lf|d_{\bbe_2^{\theta}} \bbe_2^{\theta} \rg|^2 \, e^{2\lambda} \, dy\,  dx=\frac{1}{\sin^2 \theta} \int_{\Sigma}  \lf|d_{\bbe_2^{\theta}} \bbe_2^{\theta} \rg|^2 \, dvol_g \quad.
\ee 
A straightforward computation using the definition of $\bbe_2^{\theta}$ gives
\begin{eqnarray}
\lf|d_{\bbe_2^{\theta}} \bbe_2^{\theta} \rg|^2&=& \lf|\pi_{\bn}(d_{\bbe_2^\theta} \bbe_2^{\theta})\rg|^2+ \lf(d_{\bbe_2^{\theta}} \bbe_2^{\theta}, \bbe_1 \rg)^2 + \lf(d_{\bbe_2^{\theta}} \bbe_2^{\theta}, \bbe_2 \rg)^2 \label{eq:dethetabbI} \\
&=&e^{-4\la}\lf[  \cos^4\theta \, \bbI^2_{11}+ \sin^4\theta \bbI^2_{22}+4 \sin^2 \theta \, \cos^2\theta \,\bbI^2_{12} \rg]+ \cos^2 \theta \, \lf(d_{\bbe_1} \bbe_1, \bbe_2 \rg)^2+ \sin^2 \theta \lf(d_{\bbe_2} \bbe_2, \bbe_1 \rg)^2 \quad. \nonumber
\end{eqnarray}
Combining \eqref{eq:LB01} and \eqref{eq:dethetabbI} we obtain
\be\label{eq:LB02}
\int_{\Sigma} e^{-4\la}\lf[\frac{\cos^4\theta}{\sin^2 \theta} \, \bbI^2_{11}+ \sin^2\theta \, \bbI^2_{22}+4 \cos^2\theta \,\bbI^2_{12}\rg] +\lf[ \cot^2 \theta \, \lf(d_{\bbe_1} \bbe_1, \bbe_2 \rg)^2+ \lf(d_{\bbe_2} \bbe_2, \bbe_1 \rg)^2\rg] \, dvol_g \geq \frac{4\pi}{\tau_2}.
\ee
Observing that $|d_{\bbe_1} \bbe_1|^2=e^{-4\la}\bbI ^2_{11}+ \lf(d_{\bbe_1}\bbe_1, \bbe_2 \rg)^2$ and putting together \eqref{eq:LB01} with  \eqref{eq:LB02} gives the first claim \eqref{eq:LB0}.

In order to obtain \eqref{eq:LB1}, let us recall that
\be\label{eq:bbe2bbI}
\lf|d \bbe\rg|^2:=\sum_{i,j=1}^{2} \lf|d_{\bbe_i} \bbe_j \rg|^2= e^{-4\la} \lf[ \bbI^2_{11}+\bbI^2_{22}+2 \bbI^2_{12}\rg]+ 2 \lf(d_{\bbe_1} \bbe_1 , \bbe_2\rg)^2 + 2 \lf(d_{\bbe_2} \bbe_2, \bbe_1 \rg)^2. 
\ee
Observe also that by the definition of $M$ as in \eqref{eq:defM}, we have $\theta \in [\pi/3, 2\pi/3]$, therefore we get that $4 \cos^2 \theta \leq 1$ and $1+ \cot^2 \theta\leq \frac 4 2 < 2$. Now, combining the last trigonometric estimates with \eqref{eq:LB1} and \eqref{eq:bbe2bbI},  we conclude that \eqref{eq:LB1} holds true.
\hfill$\Box$
\\

Now we can prove the lower bound (and the rigidity statement) for the frame energy of immersed tori in arbitrary codimension, namely Theorem \ref{thm:LB}. Notice the analogy with the Willmore conjecture (proved by Marques-Neves in codimension one but still open in arbitrary codimension).
\\

\textbf{Proof of Theorem \ref{thm:LB}}. First of all, thanks to Lemma \ref{Lm:RedctLB} we can assume that 
\begin{itemize}
\item the reference torus is flat (so, following the notations above, it will be denoted with $\Sigma$) and is given by the quotient of $\R^2$ via the $\Z^2$ lattice generated by the vectors $(1,0),(\tau_1,\tau_2)$ with $(\tau_1,\tau_2)\in M$ defined in $\eqref{eq:defM}$, 
\item the immersion $\bP:\Sigma\hookrightarrow \R^m$ is conformal,
\item $\bbe$ is the coordinate moving frame associated to $\bP$: $\bbe_i=\frac{\p_{x_i}\bP}{|\p_{x_i}\bP|}$.  
\end{itemize} 
Once this reduction is perfomed, we proved in Proposition \ref{Prop:LowerBound1}  that the lower bound \eqref{eq:LB1} holds, namely
\be\label{eq:PfLB1}
\F(\bP,\bbe):= \frac{1}{4} \int_\Sigma \lf|d\bbe \rg|^2\, dvol_g \geq \pi^2 \lf(\tau_2+\frac{1}{\tau_2} \rg) \lf( \frac{\sin^2 \theta}{\sin^2 \theta + \cos^4 \theta} \rg) \quad,
\ee
where $\theta=\theta_\Sigma=\arccos \tau_1$. Notice that the expression above is symmetric with respect to the map $\tau_1\mapsto -\tau_1$, so it is enough to consider $(\tau_1,\tau_2) \in M^+$, where $M^+:=M\cap \{\tau_1\geq 0\}$.

From now on we denote with $f:M^+\to \R$ the function
\be\label{eq:deff}
f(\tau_2,\theta):=\lf(\tau_2+\frac{1}{\tau_2} \rg) \lf( \frac{\sin^2 \theta}{\sin^2 \theta + \cos^4 \theta} \rg) \quad. 
\ee 
A first attempt would be to prove that $f$ is bounded below by $2$ on the whole $M^+$. However, by an elementary computation, it is easy to check  that 
$$f(\tau_2=1, \theta=\pi/2)=2 \quad\text{and}\quad f(\tau_2=\sin\theta, \theta)<2 \text{ for }\theta\in[\pi/3,\pi/2) \quad;  $$
or, in other words, except for $\tau=(0,1)$, on the arc of circle $S^1\cap M^+ $ we always have $f<2$. 
\newline

In order to overcome this difficulty let us recall that, by equation \eqref{eq:FFTW}, we can write
\be\label{eq:FFTW1}
\F(\bP,\bbe)=\F_T(\bP,\bbe)+W(\bP)
\ee
where $W(\bP)=\int_{\Sigma} |H_{\bP}|^2 \, dvol_{g_{\bP}}$ is the Willmore functional of $\bP$ and $\F_T$-defined in \eqref{def:FT}-is a non negative functional. 
Let us denote  
\be \nonumber 
\Omega_{LYMR}:=\lf\{(\tau_1,\tau_2): \lf(\tau_1-\frac{1}{2}\rg)^2+\lf(\tau_2-1\rg)^2\leq \frac{1}{4}\rg\}\cap M^+ \quad,
\ee  
and recall that if $\tau\in \Omega_{LYMR}$ the Willmore conjecture holds true (see \cite[Corollary 7]{MonRos}; this remarkable result of Montiel and Ros extend a previous celebrated result of Li and Yau \cite{LY}), namely one has
\be\label{eq:WillConj}
W(\bP)\geq 2 \pi^2 \quad \text{for every smooth conformal immersion $\bP:\R^2/(\Z\times\tau\Z)$ with $\tau\in \Omega_{LYMR}$.}
\ee
A direct computation shows that 
\be\label{eq:LBfp}
f|_{\p \Omega_{LYMR}} \geq 2 \text{ with equality if and only if } \tau_2=1 \text{ and } \theta=\frac{\pi}{2},
\ee
where, of course, $\p \Omega_{LYMR}=\lf\{(\tau_1,\tau_2): \lf(\tau_1-\frac{1}{2}\rg)^2+\lf(\tau_2-1\rg)^2= \frac{1}{4}\rg\}\cap M^+$. Observing that  the function $\tau_2 \mapsto f(\tau_2,\theta)$ is monotone strictly increasing for $\tau_2\geq 1$, the lower bound \eqref{eq:LBfp} implies that 
\be\label{eq:LBf}
f|_{M^+\setminus \Omega_{LYMR}} \geq 2 \text{ with equality if and only if } \tau_2=1 \text{ and } \theta=\frac{\pi}{2}.
\ee
The claimed lower bound for the frame energy \eqref{eq:LB} follows then by combining on one hand  \eqref{eq:PfLB1}, \eqref{eq:deff} \eqref{eq:LBf} and on the other hand \eqref{eq:FFTW1} with \eqref{eq:WillConj}.
\newline

Now let us discuss the rigidity statement. From the work of Montiel and Ros (in particular combining  Corollary 6 and the estimate (1.11) in \cite{MonRos}), we already know that
\be\label{eq:WillConjRig}
W(\bP)> 2 \pi^2 \quad \text{for every smooth conformal immersion $\bP:\R^2/(\Z\times\tau\Z)$ with $\tau\in \mathring{\Omega}_{LYMR}$,}
\ee
where $\mathring{\Omega}_{LYMR}$ is the interior of the region $\Omega_{LYMR}$ as subset of $M^+$. Therefore, combining on one hand \eqref{eq:PfLB1}, \eqref{eq:deff} \eqref{eq:LBf} and on the other hand \eqref{eq:FFTW1} with  \eqref{eq:WillConjRig} we get that  if $\F(\bP,\bbe)=2\pi^2$ then $\bP$ is a smooth conformal embedding (recall that if $\bP$ has self intersection then by \cite{LY} one has $W(\bP)\geq 8\pi$) of the \emph{flat square torus}-i.e. $\tau_2=1,\theta=\frac{\pi}{2}$- into $\R^m$. 

At this point the rigidity statement would follow by the work of Li-Yau \cite{LY}, where they prove that the Clifford torus is the unique minimizer of the Willmore energy in its conformal class. 
In any case, below, we wish to give an elementary proof of the rigidity.

Observing that the flat square torus lies in $\Omega_{LYMR}$,   we have again by \eqref{eq:WillConj} that
$$2\pi^2=\F(\bP,\bbe)=W(\bP)+\F_T(\bP,\bbe)\geq 2\pi^2+\F_T(\bP,\bbe) \quad, $$
and since $\F_T$ is non-negative we obtain
\be\label{eq:RigFT}
\F_T(\bP,\bbe)=\frac{1}{2}\int_{[0,1]^2} |\bbe_1 \cdot d\bbe_2|_g^2 \, dvol_{g_{\bP}}=0 \quad \Rightarrow \quad \bbe_1 \cdot d\bbe_2 \equiv 0 \text{ on  } [0,1]^2.
\ee 
A simple computation shows that $\bbe_1\cdot d \bbe_2=*_g d\lambda$ (which, in our setting writes more easily as $\bbe_1 \cdot \nabla \bbe_2=-\nabla^{\perp} \lambda$), where $\lambda=|\log(\p_{x_1}\bP)|$ is the conformal factor. Therefore the conformal factor of the immersion $\bP$ is constant and, up to a scaling in $\R^m$, $\bP$ is actually an isometric embedding of the square torus into $\R^m$. At this point, repeating the proof of Proposition \ref{Prop:LowerBound1} we observe that now  $\theta=0$ so the curves $\vec{\gamma}_x$ and $\vec{\gamma_y}$ are the coordinate curves. Moreover equality must hold in Fenchel Theorem \eqref{eq:Fenchel}; it follows that $\vec{\gamma}_x,\vec{\gamma}_y$ are planar convex curves with curvatures respectively $k_x(\cdot), k_y(\cdot)$. Since also in the Schwarz inequality bringing  respectively to  \eqref{eq:LB00} and \eqref{eq:LB01} there must be equality, it follows that the curvatures $k_x(\cdot), k_y(\cdot)$ are constant, so $\vec{\gamma}_x$ and $\vec{\gamma}_y$ are two planar circles of constant radius one whose plane may depend on $x$ and $y$ respectively. Finally, we claim that the plane is indenpedent of $x$ and $y$. Indeed, since $\theta=0$ and $\bbe_1\cdot d \bbe_2=\bbe_2\cdot d \bbe_1=0$ by \eqref{eq:RigFT}, the estimate \eqref{eq:LB0} reduces to
\be\label{eq:RigFlat}
2\pi^2\leq \frac{1}{4}\int_{[0,1]^2} e^{-4 \la} \lf[\bbI^2_{11}+\bbI^2_{22} \rg]\, dvol_g \quad.
\ee    
But, on the other hand, using \eqref{eq:F=FT+bbI}, \eqref{eq:RigFT}  and the energy assumption, we have  that 
\be\label{eq:RigFlat1}
2\pi^2=\F(\bP,\bbe)=\frac{1}{4}\int_{[0,1]^2} |\bbI|^2 dvol_g =\frac{1}{4}\int_{[0,1]^2} e^{-4 \la}\lf[ \bbI^2_{11}+\bbI^2_{22}+2 \bbI^2_{12} \rg]\, dvol_g \quad. 
\ee
Combining \eqref{eq:RigFlat} and \eqref{eq:RigFlat1} yields $\bbI_{12}\equiv 0$ which, combined with \eqref{eq:RigFT}, implies that 
$$\frac{d}{dx} \dot{\vec{\gamma}}_x (y)=\p^2_{x_1 x_2} \bP=\p^2_{x_2 x_1}\bP=\frac{d}{dy} \dot{\vec{\gamma}}_y(x)  \equiv 0  \quad \text{on } [0,1]^2. $$
We conclude that the plane where $\vec{\gamma}_x$ (resp. $\vec{\gamma_y}$) lies does not depend on $x$ (resp. $y$) and then, up to rotations, $\bP([0,1]^2)= S^1 \times S^1\subset \R^4 \subset \R^m$.

Let us conclude by discussing the rigidity of the frame. From the discussion of Subsection \ref{subsec:LBRed} it should be clear that if $(\bP,\bbe)$ attains the minimal value $2\pi^2$, then the frame $\bbe$ must be a Coulomb frame (this is because in particular it minimizes the tangential frame energy $\F_T$); moreover, called $\bbf:=\lf(\frac{\p}{\p \theta}, \frac{\p}{\p \varphi} \rg)$ the flat coordinate moving frame on $S^1\times S^1$,  the estimate \eqref{est:bbebbf}  gives that $\F_T(\bP,\bbe)\geq \F_T(\bP,\bbf)$ with equality if and only if $\bbe$ is a constant rotation of $\bbf$. This was exactly our claim.
\hfill$\Box$

\section{Geometric systems of conservation laws associated to the Frame energy}\label{Sec:ConsLaw}

\subsection{First variation formula for the frame energy}\label{SubSec:FirstVar}
For simplicity of presentation, and since  our applications are in in codimension one, here we present the formulas of the first variation to the frame energy for weak conformal immersions in  the euclidean three space $\R^3$; the higher codimensional computations are similar but notationally more involved and can be performed along the same lines as in \cite{Riv1}. 

Since by equation \eqref{eq:FFTW} the frame energy is the sum of the tangential frame energy $\F_T$ and of the Willmore energy $W$, and since the first variation formula for $W$ is well known (see for instance \cite{Wei} for the classical form of the equation, and \cite{Riv1}-\cite{RiCours} for the divergence form in $\R^m$, and \cite{MoRi2} for the divergence form in Riemannian manifolds), here we compute the first variation of the tangential frame energy $\F_T$. This is the content of the next proposition.   Before stating it let us introduce some notations.
\newline

Let $\bP\in {\cal E}(D^2, \R^3) $ be a weak conformal immersion, $\lambda=\log(|\p_{x_1}\bP|)=\log(|\p_{x_2}\bP|)$ the conformal factor  and $\vec{e}:=(\vec{e}_1,\vec{e}_2)=e^{-\la}(\p_{x_1}\bP,\p_{x_2}\bP)$ the associated orthonormal frame. 

For any  smooth vector field  $\vec{w}\in  C^\infty_c(\R^3,\R^3)$, we call $\bP_t(x):=\bP(x)+t\bw(\bP(x))\in {\cal E}(D^2, \R^3)$ the perturbed weak immersion and we consider the following orthonormal frame associated to $\bP_t$:
\be
\label{III.3a-1}
\vec{e}_{1,t}:=\vec{e}_2\times \vec{n}_t=\vec{e}_1+\ t\,e^{-\la}\ ( \p_{x_1} \bw, \,\bn) \; \bn+o(t),
\ee
where, in the second equality, we used \eqref{III.9};  the second vector of the frame is therefore
\be \label{III.3a-2}
\vec{e}_{2,t}=-\vec{e}_{1,t}\times \vec{n}_t=\vec{e}_2 +\ t\,e^{-\la}\ ( \p_{x_2}\bw, \,\bn) +o(t).
\ee

\begin{Prop}\label{rem:1VarFrameR3}
Let $\bP\in {\cal E}(D^2, \R^3)$ be a weak  conformal immersion,  $\lambda=\log(|\p_{x_1}\bP|)=\log(|\p_{x_2}\bP|)$ the conformal factor  and $\vec{e}:=(\vec{e}_1,\vec{e}_2)=e^{-\la}(\p_{x_1}\bP,\p_{x_2}\bP)$ the associated orthonormal frame.

Then, for any  smooth perturbation $\vec{w}\in C^\infty_c(\R^3,\R^3)$ with $\vec{w}|_{\bP(\p D^2)}=0$, called $\bP_t(x):=\bP(x)+t\bw(\bP(x))\in {\cal E}(D^2, \R^3)$ the perturbed weak immersion and $\bbe_t:=(\bbe_{1,t},\bbe_{2,t})$  the associated moving frame defined in \eqref{III.3a-1}-\eqref{III.3a-2}, it holds
\be
\label{III.3a-16R3}
\begin{array}{l}
\ds\frac{d}{dt}\F_T(\bP_t,\bbe_t)(0):=\frac{d}{dt}\lf[\frac{1}{2}\int_{\D^2}\lf<{\vec{e}}_{2,t} \cdot d{\vec{e}}_{1,t}\rg>_{g_t}^2\,dvol_{g_t}\rg](0) \\[5 mm]
\ds\qquad  \quad \qquad \qquad=\ \int_{\D^2}\ \lf(\vec{w},  d\lf[\vec{\mathbb I}\res_g(\vec{e}_2, d\vec{e}_1)\rg] +  d\ast_g\lf[\lf[ (\vec{e}_2,  d\vec{e}_1)\otimes(\vec{e}_2, d\vec{e}_1)-\ 2^{-1}\, \lf<\vec{e}_2 \cdot d\vec{e}_1\rg>_g^2\ g \rg]\res_g d\vec{\Phi}\rg]\rg) \\[5mm]
\quad \quad \quad \quad \quad \quad  \quad =\int_{\D^2}\ \lf(\vec{w},  \div \lf[-\vec{\mathbb I}\res_g(\vec{e}_2, \nabla^{\perp} \vec{e}_1)-(\bbe_2,  \nabla^\perp \bbe_1) \, \lf< (\bbe_2, \nabla \bbe_1), \nabla \bP\rg>_g+\frac{1}{2} (\bbe_2 , \nabla \bbe_1)^2 \nabla^\perp \bP\rg] \rg)\quad ,
\end{array}
\ee
where in the last formula we use the following notation: $\div$ is the divergence in $\R^2$ with euclidean metric, $\nabla^\perp=(-\p_{x_2},\p_{x_1})$, $(\bu,\bv)$ or $\bu \cdot \bv$ denotes the euclidean scalar product in $\R^3$ and $<u,v>_g$ denotes the scalar product in $(D^2,g)$, where $g=\bP^*g_{\R^3}$ is the pullback on $D^2$of the euclidean metric in $\R^3$. In the second formula, $\ast_g$ denotes the Hodge duality with respect to $g$, $d$ is the Cartan differential, and $\res_g$ is the restriction of forms with respect to $g$.
\end{Prop}

\begin{proof}
Using the expression of $\bbe_t$ given in \eqref{III.3a-1}-\eqref{III.3a-2} we compute
\[
\frac{d}{dt}\lf(\vec{e}_{2,t},d\vec{e}_{1,t}\rg)^2=\frac{d}{dt}\lf(\sum_{ij}g_t^{ij}\ (\vec{e}_{2,t},\p_{x_i}\vec{e}_{1,t})\ (\vec{e}_{2,t}, \p_{x_j}\vec{e}_{1,t})\rg) \quad.
\]
We have in one hand
\begin{eqnarray}
(\vec{e}_{2,t},\p_{x_i}\vec{e}_{1,t})&=&(\vec{e}_{2}, \p_{x_i}\vec{e}_{1})+t\,e^{-\la}\, (\p_{x_2}\bw, \bn)\ (\p_{x_i}\vec{e}_1,\vec{n})+t\,e^{-\la}\,\p_{x_1}w\ (\vec{e}_2,\p_{x_i}\vec{n})\nonumber \\
&=&(\vec{e}_{2}, \p_{x_i}\vec{e}_{1})+t\ (\p_{x_2}\bw, \bn) \ {\mathbb I}_{i1}\ e^{-2\la}-t\ (\p_{x_1}\bw, \bn)\ {\mathbb I}_{i2}\ e^{-2\la}+o(t) \quad . \label{III.3a-3}
\end{eqnarray}
Hence using (\ref{0III.10h}) we obtain
\begin{eqnarray}
\frac{d}{dt}\lf(\vec{e}_{2,t},d\vec{e}_{1,t}\rg)^2_{g_t}(0)&=&- e^{-4\la}\ \sum_{ij} \lf[ (\p_{x_i} \bP, \p_{x_j}\bw)+(\p_{x_j}\bP,\p_{x_i}\bw)\rg]\ (\vec{e}_{2}, \p_{x_i}\vec{e}_{1}) (\vec{e}_{2},\p_{x_j}\vec{e}_{1} )\nonumber \\
&&+2 e^{-4\la} \sum_i (\bbe_2, \p_{x_i}\bbe_1)  \lf[ (\p_{x_2}\bw,\bn)\,{\mathbb I}_{1i}-(\p_{x_1}\bw,\bn)\,{\mathbb I}_{2i}\rg] \nonumber \\
&=& -2 \lf< \lf[(\bbe_2,d\bbe_1)\otimes (\bbe_2,d\bbe_1) \rg] \llcorner_g d\bP, d \bw\rg>_g  -2 \lf< \bbI\llcorner_g (\bbe_2,\bbe_1), *_g d \bw \rg>_g. \label{III.3a-4} 
\end{eqnarray}
Combining now the variation of the volume form (\ref{III.3a})-computed in the appendix-and (\ref{III.3a-4}), we obtain
\begin{equation}\nonumber
2\frac{d}{dt}\F_T(\bP_t,\bbe_t)(0)= \int_{D^2}\lf[-2 \lf< \lf[(\bbe_2,d\bbe_1)\otimes (\bbe_2,d\bbe_1) -\frac{1}{2}|\bbe_2,d\bbe1|^2 g \rg] \llcorner_g d\bP, d \bw\rg>_g  -2 \lf< \bbI\llcorner_g (\bbe_2,\bbe_1), *_g d \bw \rg>_g \rg] \, dvol_g
\end{equation}
The thesis follows with an  integration by parts  recalling that by assumption $\bw|_{\bP(\p D^2)}=0$.
\end{proof}

\subsection{Some general conservation laws for conformal immersions} \label{SubSec:GenConsLaw}
The goal of the present section  is to prove the following Proposition.

\begin{Prop}\label{rem-cons-lawR3}
Let $\bP\in{\cal E}(D^2,\R^3)$ be a weak  conformal immersion,  $\lambda=\log(|\p_{x_1}\bP|)=\log(|\p_{x_2}\bP|)$ the conformal factor  and $\vec{e}:=(\vec{e}_1,\vec{e}_2)=e^{-\la}(\p_{x_1}\bP,\p_{x_2}\bP)$ the associated coordinate orthonormal frame. Then the following identities hold
 \be
 \label{a-III.3a-23}
(\ast_g d\vec{\Phi})\cdot\lf[\vec{\mathbb I}\res_g(\vec{e}_2\cdot d\vec{e}_1)+\ast_g\lf(\lf[ \vec{e}_2\cdot d\vec{e}_1\otimes\vec{e}_2\cdot d\vec{e}_1-\ 2^{-1}\, |\vec{e}_2\cdot d\vec{e}_1|^2\ g \rg]\res_g d\vec{\Phi}\rg)\rg] =0
\ee
and
\be
\label{a-III.3a-24R3}
\ds\,\lf< (\ast_g d\vec{\Phi}){\wedge}\lf[\vec{\mathbb I}\res_g(\vec{e}_2\cdot d\vec{e}_1)+\ast_g\lf(\lf[ \vec{e}_2\cdot d\vec{e}_1\otimes\vec{e}_2\cdot d\vec{e}_1-\ 2^{-1}\, |\vec{e}_2\cdot d\vec{e}_1|^2\ g \rg]\res_g d\vec{\Phi}\rg)\rg]\rg>_g=- <\ast_g d\la,d\vec{D}>_g\quad.
\ee
where
\be
\label{a-III.3a-25R3}
d\vec{D}=-{\mathbb I}\res_g\ast_gd\vec{\Phi}\wedge\vec{n}\quad.
\ee
\end{Prop}

The rest of the section is devoted to the proof of Proposition \ref{rem-cons-lawR3}. We start by computing
\be
\label{a-III.3a-17}
\begin{array}{l}
\ds (\ast d\vec{\Phi})\cdot\lf[\vec{\mathbb I}\res_g(\vec{e}_2\cdot d\vec{e}_1)+\ast\lf(\lf[ \vec{e}_2\cdot d\vec{e}_1\otimes\vec{e}_2\cdot d\vec{e}_1-\ 2^{-1}\, |\vec{e}_2\cdot d\vec{e}_1|^2\ g \rg]\res_g d\vec{\Phi}\rg)\rg]\\[5mm]
\ds\quad=\lf<\lf[ \vec{e}_2\cdot d\vec{e}_1\otimes\vec{e}_2\cdot d\vec{e}_1-\ 2^{-1}\, |\vec{e}_2\cdot d\vec{e}_1|^2\ g \rg],d\vec{\Phi}\dot{\otimes} d\vec{\Phi}\rg>_g\\[5mm]
\ds\quad=\lf<\lf[ \vec{e}_2\cdot d\vec{e}_1\otimes\vec{e}_2\cdot d\vec{e}_1-\ 2^{-1}\, |\vec{e}_2\cdot d\vec{e}_1|^2\ g \rg],g\rg>_g=0
\end{array}
\ee
where the upper dot means a contraction in the ${\R}^3$ coordinates; this gives the first part of the thesis, namely \eqref{a-III.3a-23}.  The second identity of the thesis is more subtle. We compute
\be
\label{a-III.3a-18a}
\begin{array}{l}
\ds\lf< (\ast d\vec{\Phi}){\wedge}\lf[\vec{\mathbb I}\res_g(\vec{e}_2\cdot d\vec{e}_1)+\ast\lf(\lf[ \vec{e}_2\cdot d\vec{e}_1\otimes\vec{e}_2\cdot d\vec{e}_1-\ 2^{-1}\, |\vec{e}_2\cdot d\vec{e}_1|^2\ g \rg]\res_g d\vec{\Phi}\rg)\rg]\rg>_g\\[5mm]
\quad= -\,e^{-4\la}\, \p_{x_2}\vec{\Phi}\wedge\lf(\vec{\mathbb I}_{11}\,(\vec{e}_2\cdot \p_{x_1}\vec{e}_1)+ \vec{\mathbb I}_{12}\,(\vec{e}_2\cdot \p_{x_2}\vec{e}_1) \rg) +\,e^{-4\la}\, \p_{x_1}\vec{\Phi}\wedge\lf(\vec{\mathbb I}_{21}\,(\vec{e}_2\cdot \p_{x_1}\vec{e}_1)+ \vec{\mathbb I}_{22}\,(\vec{e}_2\cdot \p_{x_2}\vec{e}_1) \rg)\\[5mm]
\quad \quad \quad+\lf<d\vec{\Phi}\wedge\lf(\lf[ \vec{e}_2\cdot d\vec{e}_1\otimes\vec{e}_2\cdot d\vec{e}_1-\ 2^{-1}\, |\vec{e}_2\cdot d\vec{e}_1|^2\ g \rg]\res_g d\vec{\Phi}\rg)\rg>_g \quad.
\end{array}
\ee
An elementary computation gives
\be
\label{a-III.3a-18b}
\lf<d\vec{\Phi}\wedge\lf(\lf[ \vec{e}_2\cdot d\vec{e}_1\otimes\vec{e}_2\cdot d\vec{e}_1-\ 2^{-1}\, |\vec{e}_2\cdot d\vec{e}_1|^2\ g \rg]\res_g d\vec{\Phi}\rg)\rg>_g=0 \quad.
\ee
Hence (for the moment the following formula is not used but it will be useful later in the section)
\be
\label{a-III.3a-18c}
\begin{array}{l}
\ds\lf< (\ast d\vec{\Phi}){\wedge}\lf[\vec{\mathbb I}\res_g(\vec{e}_2\cdot d\vec{e}_1)+\ast\lf(\lf[ \vec{e}_2\cdot d\vec{e}_1\otimes\vec{e}_2\cdot d\vec{e}_1-\ 2^{-1}\, |\vec{e}_2\cdot d\vec{e}_1|^2\ g \rg]\res_g d\vec{\Phi}\rg)\rg]\rg>_g=\lf< (\ast d\vec{\Phi}){\wedge}\vec{n},{\mathbb I}\res_g(\vec{e}_2\cdot d\vec{e}_1)\rg>_g
\end{array}
\ee
Using that $\bbe_2\cdot d\bbe_1=\ast d \la=-\frac 1 2 e^{2\la} \ast d e^{-2\la}$, from \eqref{a-III.3a-18a} we get
\be
\label{a-III.3a-18}
\begin{array}{l}
\ds\lf< (\ast d\vec{\Phi}){\wedge}\lf[\vec{\mathbb I}\res_g(\vec{e}_2\cdot d\vec{e}_1)+\ast\lf(\lf[ \vec{e}_2\cdot d\vec{e}_1\otimes\vec{e}_2\cdot d\vec{e}_1-\ 2^{-1}\, |\vec{e}_2\cdot d\vec{e}_1|^2\ g \rg]\res_g d\vec{\Phi}\rg)\rg]\rg>_g\\[5mm]
\quad=2^{-1}\,e^{-2\la}\ \p_{x_2}\vec{\Phi}\wedge\vec{n}\ \lf(-{\mathbb I}_{11}\,\p_{x_2}e^{-2\la}+{\mathbb I}_{12}\,\p_{x_1}e^{-2\la}\rg)- 2^{-1}\, e^{-2\la}\ \p_{x_1}\vec{\Phi}\wedge \vec{n}\ \lf(-{\mathbb I}_{21}\,\p_{x_2}e^{-2\la}+{\mathbb I}_{22}\,\p_{x_1}e^{-2\la}\rg)\\[5mm]
\quad=2^{-1}\,e^{-2\la}\ \p_{x_2}\vec{\Phi}\wedge\vec{n}\ \lf(-{\mathbb I}^0_{11}\,\p_{x_2}e^{-2\la}+{\mathbb I}^0_{12}\,\p_{x_1}e^{-2\la}+\,2\, H\,\p_{x_2}\la\rg)\\[5mm]
\quad\quad- 2^{-1}\, e^{-2\la}\ \p_{x_1}\vec{\Phi}\wedge \vec{n}\ \lf(-{\mathbb I}^0_{21}\,\p_{x_2}e^{-2\la}-{\mathbb I}^0_{11}\,\p_{x_1}e^{-2\la}-\,2\, H\,\p_{x_1}\la\rg) .
\end{array}
\ee
So we have established the following
\be
\label{a-III.3a-19}
\begin{array}{l}
\ds\lf< (\ast d\vec{\Phi}){\wedge}\lf[\vec{\mathbb I}\res_g(\vec{e}_2\cdot d\vec{e}_1)+\ast\lf(\lf[ \vec{e}_2\cdot d\vec{e}_1\otimes\vec{e}_2\cdot d\vec{e}_1-\ 2^{-1}\, |\vec{e}_2\cdot d\vec{e}_1|^2\ g \rg]\res_g d\vec{\Phi}\rg)\rg]\rg>_g\\[5mm]
\quad=\,2^{-1}\,e^{-2\la}\ \p_{x_2}\vec{\Phi}\wedge\vec{n}\ \lf[-\p_{x_2}({\mathbb I}^0_{11}\,e^{-2\la})+\p_{x_1}({\mathbb I}^0_{12}\,e^{-2\la})\rg]\\[5mm]
\quad\quad\, 2^{-1}\,e^{-2\la}\ \p_{x_2}\vec{\Phi}\wedge\vec{n}\ \lf[\p_{x_2}{\mathbb I}^0_{11}\,e^{-2\la}-\p_{x_1}{\mathbb I}^0_{12}\,e^{-2\la}+\, 2\,H\,\p_{x_2}\la\rg]\\[5mm]
\quad\quad-\,2^{-1}\, e^{-2\la}\ \p_{x_1}\vec{\Phi}\wedge \vec{n}\ \lf[-\p_{x_2}({\mathbb I}^0_{21}\,e^{-2\la})-\p_{x_1}({\mathbb I}^0_{11}\,e^{-2\la})\rg]\\[5mm]
\quad\quad-\,2^{-1}\, e^{-2\la}\ \p_{x_1}\vec{\Phi}\wedge \vec{n}\ \lf(\p_{x_2}{\mathbb I}^0_{21}\,e^{-2\la}+\p_{x_1}{\mathbb I}^0_{11}\,e^{-2\la}-\, 2\,H\,\p_{x_1}\la\rg) \quad.
\end{array}
\ee
Using the expression of $\vec{H}^0$ \eqref{xsIII.17} and Codazzi identity (\ref{xsIII.20.0}), we obtain

\be
\label{a-III.3a-20}
\begin{array}{l}
\ds\lf<(\ast d\vec{\Phi}){\wedge}\lf[\vec{\mathbb I}\res_g(\vec{e}_2\cdot d\vec{e}_1)+\ast\lf(\lf[ \vec{e}_2\cdot d\vec{e}_1\otimes\vec{e}_2\cdot d\vec{e}_1-2^{-1} |\vec{e}_2\cdot d\vec{e}_1|^2 g \rg]\res_g d\vec{\Phi}\rg)\rg]\rg>_g\\[5mm]
\qquad \quad=-2^{-1}e^{-2\la}\p_{x_2}\vec{\Phi}\wedge\vec{n}\lf[\p_{x_2}H^0_{\Re}+\p_{x_1}H^0_\Im\rg]-2^{-1}\p_{x_2}\vec{\Phi}\wedge\vec{n}\p_{x_2}(e^{-2\la} H)\\[5mm]
\qquad \quad\quad -2^{-1}e^{-2\la}\p_{x_1}\vec{\Phi}\wedge \vec{n}\lf[\p_{x_2}H^0_\Im-\p_{x_1}H^0_{\Re}\rg]-2^{-1}\ \p_{x_1}\vec{\Phi}\wedge \vec{n}\p_{x_1}(e^{-2\la} H) \quad.
\end{array}
\ee
Hence, using that (the first two equalities follow from \eqref{xsIII.17} and the  others from the definintion of $\mathbb I$)
\begin{eqnarray}
\p_{x_1}(e^{-2\,\la}\,\p_{x_1}\vec{\Phi})-\p_{x_2}(e^{-2\,\la}\,\p_{x_2}\vec{\Phi})=2\ \vec{H}^0_{\Re} \quad , \nonumber \\
\p_{x_1}(e^{-2\,\la}\,\p_{x_2}\vec{\Phi})+\p_{x_2}(e^{-2\,\la}\,\p_{x_1}\vec{\Phi})=-\ 2\ \vec{H}^0_\Im\quad, \nonumber \\
\p_{x_1}\vec{\Phi}\wedge\p_{x_1}\vec{n}=-\p_{x_2}\vec{\Phi}\wedge\p_{x_2}\vec{n}=-\ e^{-2\la}{\mathbb I}_{12} \,\p_{x_1}\vec{\Phi}\wedge\p_{x_2}\vec{\Phi} \quad, \nonumber \\
\p_{x_2}\vec{\Phi}\wedge\p_{x_1}\vec{n}=e^{-2\la}{\mathbb I}_{11} \, \p_{x_1}\vec{\Phi}\wedge\p_{x_2}\vec{\Phi} \quad, \nonumber\\
\p_{x_1}\vec{\Phi}\wedge\p_{x_2}\vec{n}=-\ e^{-2\la}{\mathbb I}_{22} \, \p_{x_1}\vec{\Phi}\wedge\p_{x_2}\vec{\Phi}\quad, \nonumber 
\end{eqnarray}
we get 
\be
\label{a-III.3a-21}
\begin{array}{l}
H^0_{\Im}\ \lf[\p_{x_2}\vec{\Phi}\wedge\p_{x_1}\vec{n}+\p_{x_1}\vec{\Phi}\wedge\p_{x_2}\vec{n}\rg]+ H^0_\Re\ \p_{x_2}\vec{\Phi}\wedge\p_{x_2}\vec{n}-H^0_\Re\ \p_{x_1}\vec{\Phi}\wedge\p_{x_1}\vec{n}\\[5mm]
\quad=\ e^{-2\la}\ H^0_\Im\ ({\mathbb I}_{11}-{\mathbb I}_{22})\ \p_{x_1}\vec{\Phi}\wedge \p_{x_2}\vec{\Phi}
+2\ \ e^{-2\la}\,H^0_{\Re}\ {\mathbb I}_{12}\ \p_{x_1}\vec{\Phi}\wedge \p_{x_2}\vec{\Phi}=0\quad,
\end{array}
\ee
since again $H^0_{\Re}=2^{-1}\ e^{2\la}\,({\mathbb I}_{11}-{\mathbb I}_{22})$ and $H^0_\Im=-e^{2\la}\,{\mathbb I}_{12}$.
We then deduce
\be
\label{a-III.3a-22}
\begin{array}{l}
\ds2\,\lf< (\ast d\vec{\Phi}){\wedge}\lf[\vec{\mathbb I}\res_g(\vec{e}_2\cdot d\vec{e}_1)+\ast\lf(\lf[ \vec{e}_2\cdot d\vec{e}_1\otimes\vec{e}_2\cdot d\vec{e}_1-\ 2^{-1}\, |\vec{e}_2\cdot d\vec{e}_1|^2\ g \rg]\res_g d\vec{\Phi}\rg)\rg]\rg>_g\\[5mm]
\ds\quad=-\p_{x_1}\lf[e^{-2\la}\ \p_{x_1}\vec{\Phi}\wedge[\vec{H}-\vec{H}^0_\Re]+ e^{-2\la}\ \p_{x_2}\vec{\Phi}\wedge\vec{H}^0_\Im\rg]\\[5mm]
\ds\quad \quad -\p_{x_2}\lf[ e^{-2\la}\ \p_{x_1}\vec{\Phi}\wedge\vec{H}^0_\Im+e^{-2\la}\ \p_{x_2}\vec{\Phi}\wedge[\vec{H}+\vec{H}^0_\Re]\rg]\quad .
\end{array}
\ee
Thus using (\ref{a-III.3a-18c}) we have proved so far
\be
\label{a-III.3a-22a}
\begin{array}{l}
\ds2\,\lf< (\ast d\vec{\Phi}){\wedge}\lf[\vec{\mathbb I}\res_g(\vec{e}_2\cdot d\vec{e}_1)+\ast\lf(\lf[ \vec{e}_2\cdot d\vec{e}_1\otimes\vec{e}_2\cdot d\vec{e}_1-\ 2^{-1}\, |\vec{e}_2\cdot d\vec{e}_1|^2\ g \rg]\res_g d\vec{\Phi}\rg)\rg]\rg>_g\\[5mm]
\ds\quad=2\,\lf< (\ast d\vec{\Phi}){\wedge}\vec{n},{\mathbb I}\res_g(\vec{e}_2\cdot d\vec{e}_1)\rg>_g\\[5mm]
\ds\quad=-\p_{x_1}\lf[e^{-4\la}\ \lf[\p_{x_1}\vec{\Phi}\wedge\vec{\mathbb I}_{22}-\ \p_{x_2}\vec{\Phi}\wedge\vec{\mathbb I}_{12}\rg]\rg]-\p_{x_2}\lf[ e^{-4\la}\ \lf[-\,\p_{x_1}\vec{\Phi}\wedge\vec{\mathbb I}_{12}+ \p_{x_2}\vec{\Phi}\wedge\vec{\mathbb I}_{11}\rg]\rg]\quad.
\end{array}
\ee
Now we want to use the second equality in order to write the right hand side in a more convenient way.
Observe that on one hand
\begin{eqnarray}
2\lf< (\ast d\vec{\Phi}){\wedge}\vec{n},{\mathbb I}\res_g(\vec{e}_2\cdot d\vec{e}_1)\rg>_g &=& -2\, {\mathbb I}_{21}\ \p_{x_2}\la\ \p_{x_1}\vec{\Phi}\wedge\vec{n}\ e^{-4\la}
+2\, {\mathbb I}_{22}\ \p_{x_1}\la\ \p_{x_1}\vec{\Phi}\wedge\vec{n}\ e^{-4\la}\nonumber \\
&& -2\, {\mathbb I}_{12}\ \p_{x_1}\la\ \p_{x_2}\vec{\Phi}\wedge\vec{n}\ e^{-4\la} +2\, {\mathbb I}_{11}\ \p_{x_2}\la\ \p_{x_2}\vec{\Phi}\wedge\vec{n}\ e^{-4\la}\quad. \label{a-III.3a-22b}
\end{eqnarray}
On the other hand it holds
\be
\label{a-III.3a-22c}
\begin{array}{l}
\p_{x_1}\lf[e^{-4\la}\ \lf[\p_{x_1}\vec{\Phi}\wedge\vec{\mathbb I}_{22}-\ \p_{x_2}\vec{\Phi}\wedge\vec{\mathbb I}_{12}\rg]\rg]+\p_{x_2}\lf[ e^{-4\la}\ \lf[-\,\p_{x_1}\vec{\Phi}\wedge\vec{\mathbb I}_{12}+ \p_{x_2}\vec{\Phi}\wedge\vec{\mathbb I}_{11}\rg]\rg]\\[5mm]
= e^{-2\la}\,\p_{x_1}\lf[e^{-2\la}\ \lf[\p_{x_1}\vec{\Phi}\wedge\vec{\mathbb I}_{22}-\ \p_{x_2}\vec{\Phi}\wedge\vec{\mathbb I}_{12}\rg]\rg]+e^{-2\la}\,\p_{x_2}\lf[ e^{-2\la}\ \lf[-\,\p_{x_1}\vec{\Phi}\wedge\vec{\mathbb I}_{12}+ \p_{x_2}\vec{\Phi}\wedge\vec{\mathbb I}_{11}\rg]\rg]\\[5mm]
\quad-2\, {\mathbb I}_{22}\ \p_{x_1}\la\ \p_{x_1}\vec{\Phi}\wedge\vec{n}\ e^{-4\la}+2\, {\mathbb I}_{12}\ \p_{x_1}\la\ \p_{x_2}\vec{\Phi}\wedge\vec{n}\ e^{-4\la}\\[5mm]
\quad+2\, {\mathbb I}_{21}\ \p_{x_2}\la\ \p_{x_1}\vec{\Phi}\wedge\vec{n}\ e^{-4\la}-2\, {\mathbb I}_{11}\ \p_{x_2}\la\ \p_{x_2}\vec{\Phi}\wedge\vec{n}\ e^{-4\la} \quad.
\end{array}
\ee
Combining (\ref{a-III.3a-22b}) and (\ref{a-III.3a-22c}) gives
\be
\label{a-III.3a-22d}
\begin{array}{l}
\p_{x_1}\lf[e^{-4\la}\ \lf[\p_{x_1}\vec{\Phi}\wedge\vec{\mathbb I}_{22}-\ \p_{x_2}\vec{\Phi}\wedge\vec{\mathbb I}_{12}\rg]\rg]+\p_{x_2}\lf[ e^{-4\la}\ \lf[-\,\p_{x_1}\vec{\Phi}\wedge\vec{\mathbb I}_{12}+ \p_{x_2}\vec{\Phi}\wedge\vec{\mathbb I}_{11}\rg]\rg]\\[5mm]
= e^{-2\la}\,\p_{x_1}\lf[e^{-2\la}\ \lf[\p_{x_1}\vec{\Phi}\wedge\vec{\mathbb I}_{22}-\ \p_{x_2}\vec{\Phi}\wedge\vec{\mathbb I}_{12}\rg]\rg]+e^{-2\la}\,\p_{x_2}\lf[ e^{-2\la}\ \lf[-\,\p_{x_1}\vec{\Phi}\wedge\vec{\mathbb I}_{12}+ \p_{x_2}\vec{\Phi}\wedge\vec{\mathbb I}_{11}\rg]\rg]\\[5mm]
\quad-2\,\lf< (\ast d\vec{\Phi}){\wedge}\vec{n},{\mathbb I}\res_g(\vec{e}_2\cdot d\vec{e}_1)\rg>_g \quad.
\end{array}
\ee
Thus putting together  this time the second equality of (\ref{a-III.3a-22a})  and (\ref{a-III.3a-22d}) we have obtained the following lemma
\begin{Lm}
\label{lm-identity}
For any smooth conformal immersion   the following identity holds
\be
\label{identity-1}
e^{-2\la}\,\p_{x_1}\lf[e^{-2\la}\ \lf[\p_{x_1}\vec{\Phi}\wedge\vec{\mathbb I}_{22}-\ \p_{x_2}\vec{\Phi}\wedge\vec{\mathbb I}_{12}\rg]\rg]+e^{-2\la}\,\p_{x_2}\lf[ e^{-2\la}\ \lf[-\,\p_{x_1}\vec{\Phi}\wedge\vec{\mathbb I}_{12}+ \p_{x_2}\vec{\Phi}\wedge\vec{\mathbb I}_{11}\rg]\rg]=0\quad, 
\ee
which also implies
\be
\label{identity-2}
\begin{array}{l}
 \ds\p_{x_1}\lf[e^{-2\la}\ \lf[\p_{x_1}\vec{\Phi}\wedge\vec{\mathbb I}_{22}-\ \p_{x_2}\vec{\Phi}\wedge\vec{\mathbb I}_{12}\rg]\rg]=-\p_{x_2}\lf[ e^{-2\la}\ \lf[-\,\p_{x_1}\vec{\Phi}\wedge\vec{\mathbb I}_{12}+ \p_{x_2}\vec{\Phi}\wedge\vec{\mathbb I}_{11}\rg]\rg]\quad.
\end{array}
\ee\hfill$\Box$
\end{Lm}
Let $\vec{D}$ be the two vector given by
\be
\label{identity-3R3}
\lf\{
\begin{array}{l}
\ds\p_{x_1}\vec{D}:=e^{-2\la}\ \lf[-\,\p_{x_1}\vec{\Phi}\wedge\vec{\mathbb I}_{12}+ \p_{x_2}\vec{\Phi}\wedge\vec{\mathbb I}_{11}\rg]\\[5mm]
\ds-\p_{x_2} \vec{D}:=e^{-2\la}\ \lf[\p_{x_1}\vec{\Phi}\wedge\vec{\mathbb I}_{22}-\ \p_{x_2}\vec{\Phi}\wedge\vec{\mathbb I}_{12}\rg] \quad.
\end{array} 
\rg.
\ee
Using $\vec{D}$, we can rewrite (\ref{a-III.3a-22a}) as
\be
\label{identity-4R3}
\begin{array}{l}
\ds2\,\lf< (\ast d\vec{\Phi}){\wedge}\lf[\vec{\mathbb I}\res_g(\vec{e}_2\cdot d\vec{e}_1)+\ast\lf(\lf[ \vec{e}_2\cdot d\vec{e}_1\otimes\vec{e}_2\cdot d\vec{e}_1-\ 2^{-1}\, |\vec{e}_2\cdot d\vec{e}_1|^2\ g \rg]\res_g d\vec{\Phi}\rg)\rg]\rg>_g\\
\ds\quad=\p_{x_1}\lf[e^{-2\la}\,\p_{x_2}\vec{D}\rg]-
\p_{x_2}\lf[e^{-2\la}\,\p_{x_1}\vec{D}\rg]\quad.
\end{array}
\ee
We conclude observing that
\be
\label{identity-5}
\begin{array}{l}
\p_{x_1}(e^{-2\la}\,\p_{x_2}\vec{D})-\p_{x_2}(e^{-2\la}\,\p_{x_1}\vec{D})=-2\,e^{-2\la}\,\lf[\p_{x_1}\la\,\p_{x_2}\vec{D}-\p_{x_2}\la\,\p_{x_1}\vec{D}\rg]=-2\ <\ast_g d\la,d\vec{D}>\quad.
\end{array}
\ee
Plugging \eqref{identity-5} into   \eqref{identity-4R3} we obtain  \eqref{a-III.3a-24R3}.

\subsection{A System of conservation laws involving Jacobian nonlinearities for the critical points of the  Frame energy. } \label{SubSec:SysConsLaw}

Observe that the 2-vector $\bD$  defined in \eqref{a-III.3a-25R3} (notice that $\bD$ is unique up to constants and in the following we will just be interestes in $\nabla \bD$),  using the more standard notation in $\R^3$ of vector product instead of the wedge product of vectors, can be identified (and we will do it) with the vector  defined by
\begin{equation}\label{eq:DR^3}
\nabla \bD= \bI \res_g \nabla^\perp \bP\times \bn. 
\end{equation}
Notice that $\nabla \bD \in L^2(\D^2)$.

Let us start by noticing that if the pair $(\bP,\bbe)$, where $\bP$ is a weak immersion of $\D^2$ into $\R^3$ and $\bbe$ is a moving frame on $\bP$, is a critical point for the frame energy $\F$ then, up to a reparametrization, we can assume that $\bbe$ is the coordinate moving frame associated to $\bP$, i.e. $\bbe=(\bbe_1,\bbe_2)=\frac{\sqrt{2}}{|\nabla \bP|}(\partial_{x_1}\bP, \partial_{x_2}\bP)$. This is the content of the next lemma. Let us also explicitely remark that when we say that $(\bP,\bbe)$ is a critical point of $\F$ we mean with respect to every normal perturbation of $\bP$ and any rotation of the moving frame $\bbe$.

\begin{Lm}\label{lem:reparametr}
Let  $\bP$ be a weak immersion of $\D^2$ into $\R^3$ and $\bbe$  a moving frame on $\bP$ (we stress that here $\bbe$ can be any moving frame, i.e. we do not assume a priori that $\bbe=(\bbe_1,\bbe_2)=\frac{\sqrt{2}}{|\nabla \bP|}(\partial_{x_1}\bP, \partial_{x_2}\bP)$). Assume that  $(\bP,\bbe)$ is critical for the frame energy $\F$. Then there exists a bilipshitz diffeomorphism $\psi:\D^2 \to \D^2$ such that the new weak immersion $\tilde{\bP}:=\bP\circ \psi$ is conformal and $\bbe$ is the coordinate moving frame associated to $\tilde{\bP}$, i.e. $\bbe=(\bbe_1,\bbe_2)=\frac{\sqrt{2}}{|\nabla \tilde{\bP}|}(\partial_{x_1}\tilde{\bP}, \partial_{x_2}\tilde{\bP})$. 
\end{Lm}

\begin{proof}
For the moment fix $\bP$. From the criticality of the moving frame $\bbe$ for the frame energy $\F$- or equivalently for the tangential frame energy $\F_T$-with respect to rotations (i.e. with respect to variations of the type $\bbe_t=e^{it\theta}\bbe$ for $\theta \in [0,2\pi]$), a simple computation shows that the frame $\bbe$ satisfies the Coulomb condition $\div(\bbe_1\cdot \nabla\bbe_2)=0$. At this point the existence of the reparametrization $\psi$ can be performed using the so called Chern moving frame methos and it is well known (see for instance \cite{RiCours}). 
\end{proof}

From now on, if $(\bP,\bbe)$ is a critical point of the frame energy $\F$, we will always assume that $\bbe$ is the coordinate orthonormal frame associated to $\bP$;  this is not restrictive, up to bilipschitz reparametrizations of $\bP$, thanks to Lemma \ref {lem:reparametr}. Therefore when saying that $\bP$ is a critical point of $\F$ we will mean that $(\bP,\bbe)$ is critical, where $\bbe$ is the coordinate orthonormal frame associated to $\bP$.

We remind that by identity \eqref{eq:FFTW}, we have $\F(\bP,\bbe)=\F_T(\bP,\bbe)+W(\bP)$; combining the first variation formula of the tangential frame energy $\F_T$ computed in \eqref{III.3a-16R3}
 with the first variation of the Willmore functional $W$ we obtain the one of $\F$. Let us recall that the first variation of the Willmore functional $W(\bP):=\int_{\bP} H^2 dvol_g$ on a weak conformal immersion $\bP$ of a disk $\D^2$ into $\R^3$ with respect a smooth vector field  $ \bw \in C^\infty_c(\R^3,\R^3)$ with $\bw|_{\bP(\p D^2)}=0$ is given by (for more detail see \cite{RiCours})
\begin{equation}\label{eq:1VarWR3}
\frac{d}{dt}_{|t=0} W(\bP+t\bw)=\int_{\D^2} \frac{1}{2}\div \lf[\nabla \bH -3 \nabla H  \, \bn + \nabla^\perp\bn \times \bH  \rg] \cdot \bw \quad,  
\end{equation} 
so that, we obtain
\begin{eqnarray}\label{eq:1VarFR3}
\frac{d}{dt}_{|t=0} \F(\bP+t \bw)&=&\int_{\D^2} \div\Big[ \frac{1}{2} \lf(\nabla \bH -3 \nabla H  \, \bn + \nabla^\perp\bn \times \bH \rg)  \\
&& \qquad \quad\;   -\vec{\mathbb I}\res_g(\vec{e}_2 \cdot \nabla^{\perp} \vec{e}_1)-\bbe_2 \cdot \nabla^\perp \bbe_1 \, (\bbe_2 \cdot \nabla \bbe_1, \nabla \bP)_g+\frac{1}{2} |\bbe_2 \cdot \nabla \bbe_1|_g^2 \nabla^\perp \bP\Big] \cdot \bw \quad.\nonumber
\end{eqnarray} 
So we obtained the following proposition.

\begin{Prop}
Let $\bP$ be a  weak conformal immersion  of the disk $\D^2$ into $\R^3$. Then $\bP$ is a critical point of the frame energy $\F$ if and only if 
\be\label{eq:ELeqFR3}
\div\Big[ \frac{1}{2} \lf(\nabla \bH -3 \nabla H  \, \bn + \nabla^\perp\bn \times \bH \rg)   -\vec{\mathbb I}\res_g(\vec{e}_2 \cdot \nabla^{\perp} \vec{e}_1)-\bbe_2 \cdot \nabla^\perp \bbe_1 \, (\bbe_2 \cdot \nabla \bbe_1, \nabla \bP)_g+\frac{1}{2} |\bbe_2 \cdot \nabla \bbe_1|_g^2 \nabla^\perp \bP\Big]=0 
\ee
\end{Prop}

\begin{Lm}\label{lm:LR3}
Let $\bP$ be a weak conformal immersion of the disc $D^2$ into $\R^3$ critical for the frame energy $\F$. Then there exists a vector field $\bL_F\in L^{2,\infty}_{loc}(\D^2)$ with $\nabla \bL_F \in L^1(D^2)$ such that
\begin{eqnarray}
\nabla^\perp \bL_F&=&\frac{1}{2} \lf(\nabla \bH -3 \nabla H  \, \bn + \nabla^\perp\bn \times \bH \rg) \nonumber \\
&&-\vec{\mathbb I}\res_g(\vec{e}_2 \cdot \nabla^{\perp} \vec{e}_1)-\bbe_2 \cdot \nabla^\perp \bbe_1 \, (\bbe_2 \cdot \nabla \bbe_1, \nabla \bP)_g+\frac{1}{2} |\bbe_2 \cdot \nabla \bbe_1|_g^2 \nabla^\perp \bP \quad .\nonumber
\end{eqnarray}
\end{Lm}

\begin{proof}
By the first variation formula 	\eqref{eq:1VarFR3}, we know that if $\bP$ is critical for the frame energy then
\begin{eqnarray}
&&\div\Big[ \frac{1}{2} \lf(\nabla \bH -3 \nabla H  \, \bn + \nabla^\perp\bn \times \bH \rg)  \\
&&\qquad   -\vec{\mathbb I}\res_g(\vec{e}_2 \cdot \nabla^{\perp} \vec{e}_1)-\bbe_2 \cdot \nabla^\perp \bbe_1 \, (\bbe_2 \cdot \nabla \bbe_1, \nabla \bP)_g+\frac{1}{2} |\bbe_2 \cdot \nabla \bbe_1|_g^2 \nabla^\perp \bP\Big] =0 \quad.\nonumber
\end{eqnarray}
So by the weak Poincar\'e Lemma (more precisely, one takes  successively the convolution of
the $^\perp$ of the divergence free quantity above  with the Poisson kernel $(2\pi)^{-1} \log r$, then taking
the divergence and finally subtracting some harmonic vector field one gets the conclusion; for more details see the beginning of the  proof of \cite[Theorem VII.14]{RiCours0}; in particular, the $L^{2,\infty}$ regularity of $\bL_F$ follows by a classical result of Adams \cite{Ad} on Riesz potentials) there exists $\bL_F$ as desired.
\end{proof}

\begin{Lm}\label{lm:SystCSR322}
Let $\bP$ be a weak conformal immersion of the disc $D^2$ into $\R^3$ critical for the frame energy $\F$ and let $\bL_F$ given by Lemma \ref{lm:LR3}. Then the following system of conservation laws holds:
\be
\label{eq:SystCSR322}
\lf\{
\begin{array}{l}
\div\lf(\lf< \bL_F, \nabla^\perp \bP  \rg>_{\R^3} \rg)=0 \\[5mm]
\div \lf( \bL_F \times \nabla^\perp \bP  +  \la \;\nabla^\perp \bD + H \nabla^\perp \bP\rg)=0 \quad.
\end{array}
\rg.
\ee
Therefore there exists  a function $S_F\in W^{1,(2,\infty)}_{loc}(\D^2,\R)$ and a vector field $\bR_F \in W^{1,(2,\infty)}_{loc}(\D^2,\R^3)$ satisfying
\begin{eqnarray}
\nabla S_F &=& \lf< \bL_F, \nabla \bP  \rg>_{\R^3} \label{eq:SFR322} \\
\nabla \bR_F &=&  \bL_F \times \nabla \bP  + (\la-\bar{\la}) \; \nabla \bD-H \nabla \bP \quad, \label{eq:RFR322}
\end{eqnarray}
where $\bar{\lambda}$ is the mean value of $\lambda$ on $D^2$.
\end{Lm}

\begin{proof}
Combining Remark \ref{rem-cons-lawR3}, the definition of $\bL_F$ in Lemma \ref{lm:LR3}, and Theorem VII.14  in \cite{RiCours0} (in particular see equations (VII.187) and (VII.204)),  we have
\be\nonumber
\lf\{
\begin{array}{l}
\lf< \nabla \bL_F, \nabla^\perp \bP  \rg>_g =0 \\[5mm]
\lf< \nabla \bL_F \times \nabla^\perp \bP  \rg>_g + <\nabla \la, \nabla^\perp \bD>_g + <\nabla H, \nabla^\perp \bP>_g=0 \quad.
\end{array}
\rg.
\ee
Since $\div\circ \nabla^\perp \equiv 0$, the last system is equivalent to the desired system \eqref{eq:SystCSR322}. The existence and the regularity of $S_F$ and $\bR_F$ is analogous to the existence of $\bL_F$ in Lemma \ref{lm:LR3}.
\end{proof}

\begin{Prop}\label{prop:SysRSFR322}
Let $\bP$ be a weak conformal immersion of the disc $D^2$ into $\R^3$ critical for the frame energy $\F$ and let $S_F \in W^{1,(2,\infty)}_{loc}(\D^2,\R), \bR_F\in W^{1,(2,\infty)}(\D^2,\R^3)$ be  given by Lemma \ref{lm:SystCSR322}. Then their gradients satisfy the following system:
\be\label{eq:SysgradSFRFR322}
\lf\{
\begin{array}{l}
\nabla S_F = -  \lf<\nabla^\perp \bR_F, \bn \rg> + (\la-\bar{\la}) \; \lf<\nabla^\perp \bD, \bn\rg> \\[5mm]
\nabla \bR_F=    \bn \times\nabla^\perp \bR_F + \nabla^\perp S_F \; \bn +  (\la-\bar{\la}) \; \nabla\bD +  (\la-\bar{\la}) \lf( \lf<\nabla^\perp\bD, \bbe_2 \rg> \bbe_1 -  \lf<\nabla^\perp\bD, \bbe_1 \rg> \bbe_2\rg)\quad.
\end{array}
\rg.
\ee
Therefore $(S_F,\bR_F,\bD,\bP,\la)$ satisfy the following elliptic system  
\be\label{eq:SysDeltaSFRFR322}
\lf\{
\begin{array}{l}
\Delta S_F= -  \lf<\nabla^\perp \bR_F, \nabla \bn \rg> + \div \left[ (\la-\bar{\la}) \; \lf<\nabla^\perp \bD, \bn\rg> \right] \\[5mm]
\Delta \bR_F=  \nabla \bn \times\nabla^\perp \bR_F  + \nabla^\perp S_F \;  \nabla \bn\\ [3mm]
            \qquad \qquad + \div \left[ (\la-\bar{\la}) \; \nabla\bD +  (\la-\bar{\la}) \lf( \lf<\nabla^\perp\bD, \bbe_2 \rg> \bbe_1 -  \lf<\nabla^\perp\bD, \bbe_1 \rg> \bbe_2\rg) \right] \\[5mm]
\Delta \bD = \div \lf( \bI \res_g \nabla^\perp \bP \times \bn \rg) \\[5mm]						
[1-(\la-\bar{\la})] \, \Delta \bP = - \lf< \nabla \bR_F \times \nabla^\perp \bP\rg> - \lf<\nabla S_F, \nabla^\perp \bP \rg> \\[5mm]
\Delta \lambda = - \lf<\nabla^\perp \bbe_1, \nabla \bbe_2 \rg>	\quad.				
\end{array}
\rg.
\ee
As a consequence, we have that $S_F\in W^{1,2}_{loc}(\D^2,\R)$ and $\bR_F \in W^{1,2}_{loc}(\D^2,\R^3)$.
\end{Prop}

\begin{Rm}
A natural question arising from Proposition \ref{prop:SysRSFR322} is if actually the system \eqref{eq:SysDeltaSFRFR322} is equivalent to the frame energy equation \eqref{eq:ELeqFR3}; in analogy with the situation in the Willmore framework (see \cite{BR}-\cite{RiCours0}-\cite{Riv3}) we expect this not to  be the case. More precisely we expect the system   \eqref{eq:SysDeltaSFRFR322} to be equivalent to the conformal-constrained  Willmore equation. A second observation is that we expect the conservation laws on $S_F$ and $\bR_F$ to be the associated, via Noether's Theorem, to dilations and rotations (transformations that preserve the frame energy, as observed in the introduction); this remark in the context of Willmore surfaces is due to Yann Bernard. We will study these questions in a forthcoming work. 
\end{Rm}

\begin{proof}
Recall that we use the notation
\be\label{eq:e1xe22}
\bbe_1 \times \bbe_2 = \bn, 
\ee
so, taking the scalar product in $\R^3$ between \eqref{eq:RFR322} and $\bn$ and observing that
\begin{eqnarray}
\lf<\p_{x_1} \bR_F, \bn \rg>&=& e^\la \lf<\bL_F,\bbe_2\rg> \lf<\bbe_2 \times\bbe_1, \bn \rg>+  (\la-\bar{\la}) \; \lf<\p_{x_1}\bD, \bn\rg> = -  \lf<\bL_F,\p_{x_2} \bP\rg> + (\la-\bar{\la}) \; \lf<\p_{x_1}\bD, \bn\rg> \nonumber \\
\lf<\p_{x_2} \bR_F, \bn \rg>&=& e^\la \lf<\bL_F,\bbe_1\rg> \lf<\bbe_1 \times\bbe_2, \bn \rg>+ (\la-\bar{\la}) \; \lf< \p_{x_2}\bD, \bn\rg> =   \lf<\bL_F,\p_{x_1} \bP\rg> +  (\la-\bar{\la}) \; \lf<\p_{x_2}\bD, \bn\rg> \nonumber
\end{eqnarray}
we get
\be\label{eq:nablaRfnR322}
\lf<\nabla \bR_F, \bn \rg> =   \lf<\bL_F, \nabla^\perp \bP\rg> +  (\la-\bar{\la}) \; \lf< \nabla \bD, \bn\rg>; 
\ee
recalling \eqref{eq:SFR322} and the fact that $(\nabla^\perp)^\perp=-\nabla$, the last identity gives
\be\label{eq:nablaSfR322}
\nabla S_F = -  \lf<\nabla^\perp \bR_F, \bn \rg> +  (\la-\bar{\la}) \; \lf<\nabla^\perp \bD, \bn\rg> .
\ee
Analogously one computes $\lf<\nabla \bR_F, \bbe_i \rg>$ which, combined with \eqref{eq:nablaRfnR322},  gives
\be\label{eq:nablaRfLfprovv22}
\nabla \bR_F=  \lf<\bL_F, \nabla^\perp \bP\rg> \bn - \lf<\bL_F, \bn \rg> \nabla^\perp \bP +  (\la-\bar{\la}) \; \nabla\bD \quad.
\ee
Taking the vector product of \eqref{eq:nablaRfLfprovv22}$^\perp$ with $\bn$ gives
\be\label{nablaperpRnR322}
\nabla^\perp \bR_F \times \bn = \lf< \bL_F, \bn \rg> \nabla^\perp \bP +  (\la-\bar{\la}) \lf(\lf<\nabla^\perp\bD, \bbe_2 \rg> \bbe_1- \lf<\nabla^\perp\bD, \bbe_1 \rg> \bbe_2  \rg),
\ee
which, plugged in \eqref{eq:nablaRfLfprovv22} together with \eqref{eq:SFR322}, gives
\be\nonumber
\nabla \bR_F=    \bn \times\nabla^\perp \bR_F + \nabla^\perp S_F \; \bn + (\la-\bar{\la}) \; \nabla \bD+ (\la-\bar{\la}) \lf(\lf<\nabla^\perp\bD, \bbe_2 \rg> \bbe_1- \lf<\nabla^\perp\bD, \bbe_1 \rg> \bbe_2  \rg)\quad.
\ee
Applying the divergence and recalling that $\div(\nabla^\perp)\equiv 0$ we get the first two equations.
The very definition of $\bD$ gives the third equation.  In order to obtain the fourth equation compute  the vector product between $\nabla^\perp \bP$ and $\bR_F$ as in  \eqref{eq:RFR322}:
\begin{equation}
\lf<\nabla^\perp \bP \times \nabla \bR_F\rg> =  \lf< \nabla^\perp \bP \times({\bL}_F \times \nabla \bP) \rg>  + (\la-\bar{\la}) \;\lf<\nabla^\perp \bP \times \nabla \bD \rg>  \quad. \label{eq:NpPNR322}
\end{equation}
A short computation gives
\begin{eqnarray}
\lf<\nabla^\perp \bP \times \nabla \bD \rg>&=&\lf<\nabla^\perp \bP \times \lf( \bI{\res_g} \nabla^\perp \bP \times \bn \rg) \rg> \nonumber \\
                                           &=&  \bI{\res_g} \nabla^\perp \bP \, \lf< \nabla^\perp \bP, \bn\rg>- \bn  \lf< \bI{\res_g} \nabla^\perp \bP, \nabla^\perp \bP\rg> \nonumber\\
																					&=&-2 e^{2\la} H \bn= - \Delta \bP \quad.\label{eq:NpPtND22} 																	
\end{eqnarray}
Observing that 
$$H \lf<\nabla^\perp \bP\times \nabla \bP\rg>= 2  e^{2\la} \bH= \Delta \bP,$$
and 
$$\lf< \nabla^\perp \bP \times({\bL}_F \times \nabla \bP) \rg>= \bL_F \lf<\nabla^\perp \bP, \nabla \bP \rg>-\nabla \bP \lf< \nabla^\perp \bP, \bL_F \rg>=-\lf<\nabla \bP, \nabla^\perp S_F \rg>\quad, $$
where in the last equality we recalled \eqref{eq:SFR322} and that of of course $\lf<\nabla^\perp \bP, \nabla \bP \rg>=0$.Therefore  we can rewrite \eqref{eq:NpPNR322} as 
$$[1-(\la-\bar{\la})] \, \Delta \bP = - \lf< \nabla \bR_F \times \nabla^\perp \bP\rg> - \lf<\nabla S_F, \nabla^\perp \bP \rg> \quad .  $$
The last equation is classical, see for instance \cite[(VII.132)]{RiCours0}. 

The  improved regularity of $\la, S_F$ and $\bR_F$ is quite standard, in any case let us briefly sketch the proof for $\la$ and $S_F$ (the one for $\bR_F$ is analogous). The fact that $\la \in L^\infty(D^2)$  follows directly from the Wente-type estimate of Chanillo-Li \cite{ChLi} (for more details see also \cite[Section VII.6.4]{RiCours0}). Let us discuss the regularity of $S_F$: take any ball $B\subset D^2$ and  write $S_F$ on $B$ as
$$S_F=S_0+S_1+S_2\quad ,$$
where $S_0\in C^\infty(B)$ is the harmonic extension to $B$ of  $S_F|_{\p B}$ and $S_1,$ $S_2$ satisfy the equations
\be\label{eq:S1}
\lf\{
\begin{array}{l}
\Delta S_1= -  \lf<\nabla^\perp \bR_F, \nabla \bn \rg>  \quad \text{on} \; B\\[5mm]
S_1= 0 \quad \text{on } \p B	 \quad,	
\end{array}
\rg.
\ee
\be\label{eq:S2}
\lf\{
\begin{array}{l}
\Delta S_2= -   \div \left[ (\la-\bar{\la}) \; \lf<\nabla^\perp \bD, \bn\rg> \right] \quad \text{on} \; B\\[5mm]
S_2= 0 \quad \text{on } \p B	 \quad.
\end{array}
\rg.
\ee
Since by construction $\bR_F\in W^{1,(2,\infty)}(B)$ and $\bn \in W^{1,2}(B)$, by a refinement of the Wente inequality due to Bethuel \cite{Bet} (which is based on previous results of Coifman-Lions-Meyer and Semmes), equation \eqref{eq:S1}  implies that $S_1 \in W^{1,2}(B)$. The fact that $S_2 \in W^{1,2}(B)$ follows instead from Stampacchia gradient estimates, recalling that $\la-\bar{\la}\in L^\infty(B,\R), \bD \in W^{1,2}(B,\R^3)$ and of course $\bn\in L^\infty(B,\R^3)$.

\end{proof}

\section{Regularity of critical points for the frame energy: proof of Theorem \ref{thm:regularity}}\label{Sec:Reg}
The goal of the next Section  \ref{Sec:RegHom} is to minimize the frame energy in each  regular homotopy class  of immersed tori in $\R^3$ and to propose such a minimizer as canonical rapresentant for its own  class; to this aim, in the present section we develop the regularity theory for critical points of the frame energy. The fundamental starting point is given by the elliptic system with quadratic jacobian non linearities  satisfied by the critical points of the frame energy, namely Proposition \ref{prop:SysRSFR322}.

The strategy is to show that, for every $x_0\in \D^2$ and $r>0$ small enough, the system \eqref{eq:SysDeltaSFRFR322} with Dirichelet boundary condition has at most one solution in a suitable function space for every $0<\rho\leq r$; then, using a good slicing argument together with properties of the trace and harmonic extension, we construct a more regular solution of the system  \eqref{eq:SysDeltaSFRFR322} having the same boundary condition as the initial solution. By the uniqueness we  infer that the initial solution had to be more regular, namely in a subcritical space, then we conclude with a standard bootstrap argument that the initial  solution is actually $C^\infty$. 

\

Before stating the lemmas let us introduce some notation. 
\\In the following,  $\bP$ will  be a weak conformal immersion of $D^2$ into $\R^3$ critical for the functional $\F$. For any $\rho\in (0,1), k \in \R$ and $p\in (1,\infty)$   we will denote 
\begin{eqnarray}\nonumber 
\mathcal{E}^{k,p}(B_\rho(0))&:=&W^{k,p}(B_\rho(0),\R)\times W^{k,p}(B_\rho(0),\R^3)\times W^{k,p}(B_\rho(0),\R^3)\times W^{k+1,p}(B_\rho(0),\R^3),\nonumber \\
\mathcal{E}^{1,p}_0(B_\rho(0))&:=&W^{1,p}_0(B_\rho(0),\R)\times W^{1,p}_0(B_\rho(0),\R^3)\times W^{1,p}_0(B_\rho(0),\R^3)\times (W^{2,p}\cap W^{1,p}_0(B_\rho(0),\R^3))\nonumber,
\end{eqnarray}
Let us define  
\be\label{eq:defL}
\lf\{
\begin{array}{l}
L (A):=\Delta A +  \lf<\nabla^\perp \bB, \nabla \bn \rg> - \div \left[ (\la-\bar{\la}) \; \lf<\nabla^\perp \bC, \bn\rg> \right] \\[5mm]
L (\bB):= \Delta \bB - \nabla \bn \times\nabla^\perp \bB  - \nabla^\perp A \;  \nabla \bn\\ [3mm]
            \qquad \qquad - \div \left[ (\la-\bar{\la}) \; \nabla\bC +  (\la-\bar{\la}) \lf( \lf<\nabla^\perp\bC, \bbe_2 \rg> \bbe_1 -  \lf<\nabla^\perp\bC, \bbe_1 \rg> \bbe_2\rg) \right] \\[5mm]
L(\bC):= \Delta \bC - \div \lf( \pi_{\bn}(\nabla^2 \bPsi)  \res_g \nabla^\perp \bP \times \bn \rg) \\[5mm]
L(\bPsi):= \Delta \bPsi   +\frac{1}{1-(\la-\bar{\la})} \lf[\lf< \nabla \bB \times \nabla^\perp \bP\rg> + \lf<\nabla A, \nabla^\perp \bP \rg> \rg] \quad ,
\end{array}
\rg.
\ee
where $\bP,\bbe_{i},\bn, \lambda,\bar{\la}$ are the critical weak conformal immmersion, its normalized tangent vectors, its normal vector, its conformal factor and the mean value of the conformal factor on $B_\rho(0)$; all this terms are seen in system \eqref{eq:defL} as coefficients. 
Since $\lambda$ satisfies the elliptic equation
$$\Delta \lambda = - \lf<\nabla^\perp \bbe_1, \nabla \bbe_2 \rg>, $$
from the Wente-type estimate of Chanillo-Li \cite{ChLi} (for more details see also \cite[Section VII.6.4]{RiCours0}) we have
\begin{eqnarray}
\|\lambda-\bar{\la}\|_{L^\infty(B_\rho(0))} &\leq & C \|\nabla e_1\|_{L^2(B_\rho(0))}  \|\nabla e_2\|_{L^2(B_\rho(0))} \nonumber
\\                                                                          &\leq & C' \lf[ \int_{B_\rho(0)} |\bbe_2 \cdot \nabla \bbe_1|^2 \, dx  +  \int_{B_\rho(0)} (|\bn \cdot \nabla \bbe_1|^2+ |\bn \cdot \nabla \bbe_2|^2) \, dx  \rg]\nonumber
\\                                                                          &\leq & C'' \lf[ \int_{B_\rho(0)} |\bbe_2 \cdot d \bbe_1|_g^2 \, dvol_g  + \int_{B_\rho(0)} | \bbI |_g^2 \, dvol_g   \rg] \quad ,\label{eq:estla-labar}
\end{eqnarray}
for some $C,C',C''>0$ independent of $\bP$. 

Observe that combining  Lemma \ref{lem:Appendix(p,p)} and  \eqref{eq:estla-labar}, we get that $L$  is a linear continuous operator from  ${\mathcal E}^{1,p}(B_\rho(0))$ to ${\mathcal E}^{-1,p}(B_\rho(0))$  for every $p\in(1,\infty)$.

The  key technical lemma for proving the regularity is the following isomorphism result.

\begin{Lm}[$L$ is an isomorphism between ${\mathcal E}^{1,p}_0$ and ${\mathcal E}^{-1,p}$]\label{lem:isomorph}
Let $\bP$ be a weak conformal immersion of $D^2$ into $\R^3$ critical for the frame energy $\F$. Then there exists $r>0$ (depending on $\bP$) such that for every $\rho \in (0,r]$ the linear operator  $L$ is an isomorphism from ${\mathcal E}^{1,p}_0(B_\rho(0))$ onto ${\mathcal E}^{-1,p}(B_\rho(0))$, for every $p\in(1,\infty)$. In particular if $\bE \in {\mathcal E}^{1,p}_0(B_\rho(0))$, for some $p\in(1,\infty)$, solves the homogeneous equation $L(\bE)=0$, then $\bE=0$ a.e. on $B_\rho(0)$.
\end{Lm}

\begin{proof}
From the discussion above we already know that $L:{\mathcal E}^{1,p}_0(B_\rho(0)) \to {\mathcal E}^{-1,p}(B_\rho(0))$ is a continuous linear operator. 
Our goal is to prove that $L:{\mathcal E}^{1,p}_0(B_\rho(0)) \to {\mathcal E}^{-1,p}(B_\rho(0))$ is an isomorphism, more precisely we prove  that there exists $r>0$ such that for every $\rho\in (0,r)$ and  for every $(f_A,\vec{f}_{\bB}, \vec{f}_{\bC},\vec{f}_{\bPsi}) \in {\mathcal E}^{-1,p}(B_\rho(0))$, the system $L(A^0,\bB^0,\bC^0,\bPsi^0)=(f_A,\vec{f}_{\bB}, \vec{f}_{\bC},\vec{f}_{\bPsi})$ has a unique solution $(A^0,\bB^0,\bC^0,\bPsi^0)\in {\mathcal E}^{1,p}_0(B_\rho(0))$.

First of all observe that for every  $\delta_0>0$ to be fixed later there exists $r>0$ such that
\be\label{eq:defeps0}
\int_{B_r(0)} |\bbe_2 \cdot d \bbe_1|_g^2 \, dvol_g  + \int_{B_r(0)} | \bbI |_g^2 \, dvol_g  \leq \delta_0 \quad. 
\ee 
From now on let $\rho\in (0,r]$. Observe that, thanks to \eqref{eq:defeps0} and \eqref{eq:estla-labar}, for any $\sigma>0$ there exists $\delta_0>0$ small enough such that 
\be \label{eq:la-labar}
\|\la-\bar{\la}\|_{L^\infty(B_\rho(0))}\leq \sigma^2. 
\ee

As first step we establish a priori estimates on the solutions.
\\If  $(A^0,\bB^0,\bC^0,\bPsi^0)\in {\mathcal E}^{1,p}_0(B_\rho(0))$ solve $L(A^0,\bB^0,\bC^0,\bPsi^0)=(f_A,\vec{f}_{\bB}, \vec{f}_{\bC},\vec{f}_{\bPsi})$ then, using  \eqref{eq:la-labar} and Lemma \ref{lem:Appendix(p,p)}, and  the classical Stampacchia gradient estimates for elliptic PDEs, we can estimate (all the norms are computed on $B_\rho(0)$)
\begin{eqnarray}
\|\nabla^2 \bPsi^0\|_{L^p}&\leq& \gamma \lf[ \|\vec{f}_{\bPsi}\|_{L^p}+ \|\nabla A^0\|_{L^p}+ \|\nabla \bB^0\|_{L^p} \rg] \nonumber \\ 
\|\nabla \bC^0\|_{L^p}&\leq& \gamma \lf[ \|\nabla^2 \bPsi^0\|_{L^p}+ \|\nabla \bB^0\|_{L^p} \rg]  \nonumber\\ 
\|\nabla \bB^0\|_{L^p}&\leq& \gamma \lf[ \|\vec{f}_{\bB}\|_{W^{-1,p}}+ \varepsilon_0  \lf(\| \nabla A^0\|_{L^p}+ \|\nabla \bB^0\|_{L^p} +\|\nabla \bC^0\|_{L^p}\right)   \right] \nonumber \\
\|\nabla \bA^0\|_{L^p}&\leq& \gamma \lf[ \|\vec{f}_{A}\|_{W^{-1,p}}+ \varepsilon_0  \lf(\| \nabla A^0\|_{L^p}+ \|\nabla \bB^0\|_{L^p} +\|\nabla \bC^0\|_{L^p}\right)   \right]  \nonumber\quad, 
\end{eqnarray}
for some constant $\gamma>0$. Bootstrapping the estimates above we obtain that 
$$\|(A^0,\bB^0,\bC^0,\bPsi^0)\|_{{\mathcal E}^{1,p}_0(B_\rho(0))} \leq \gamma' \lf[ \varepsilon_0 \|(A^0,\bB^0,\bC^0,\bPsi^0)\|_{{\mathcal E}^{1,p}_0(B_\rho(0))}+ \|(f_A,\vec{f}_{\bB}, \vec{f}_{\bC},\vec{f}_{\bPsi})\|_{{\mathcal E}^{-1,p}_0(B_\rho(0))} \rg].$$
Choosing $\rho>0$ small enough such that $\gamma' \varepsilon_0\leq \frac 12$, the last estimate  gives 
\be\label{eq:aPrioriEst}
\|(A^0,\bB^0,\bC^0,\bPsi^0)\|_{{\mathcal E}^{1,p}_0(B_\rho(0))}\leq \gamma''\|(f_A,\vec{f}_{\bB}, \vec{f}_{\bC},\vec{f}_{\bPsi})\|_{{\mathcal E}^{-1,p}_0(B_\rho(0))}\quad .
\ee
Since $L$ is linear, of course, the a priori estimate \eqref{eq:aPrioriEst} ensures \emph{uniqueness} of the solution to the system $L(A^0,\bB^0,\bC^0,\bPsi^0)=(f_A,\vec{f}_{\bB}, \vec{f}_{\bC},\vec{f}_{\bPsi})$ in the space ${\mathcal E}^{1,p}_0(B_\rho(0))$. 
\newline

We now construct the solution by iteration. Given $(f_A,\vec{f}_{\bB}, \vec{f}_{\bC},\vec{f}_{\bPsi})\in {\mathcal E}^{-1,p}_0(B_\rho(0))$, let $(A_0,\bB_0,\bC_0,\bPsi_0)\in  {\mathcal E}^{1,p}_0(B_\rho(0))$ be the solution to 
\be\nonumber
\lf\{
\begin{array}{l}
\Delta A_0=f_A \\[5mm]
\Delta \bB_0=\vec{f}_{\bB} \\ [5mm] 
\Delta \bPsi_0 = \frac{1}{(\la-\bar{\la})} \lf[\lf< \nabla \bB_0 \times \nabla^\perp \bP\rg> + \lf<\nabla A_0, \nabla^\perp \bP \rg> \rg] +\vec{f}_{\bPsi} \\[5mm]
\Delta \bC_0=  \div \lf( \pi_{\bn}(\nabla^2 \bPsi_0)  \res_g \nabla^\perp \bP \times \bn \rg)+ \vec{f}_{\bC}\quad.
\end{array}
\rg.
\ee
In order to define the iteration, let us denote with $\tilde{L}:{\mathcal E}^{1,p}_0(B_\rho(0))\to {\mathcal E}^{-1,p}_0(B_\rho(0)) $ the linear operator such that $L=\Delta-\tilde{L}$ (recall the definition of $L$ in \eqref{eq:defL}). Now we define  $(A_{i+1},\bB_{i+1},\bC_{i+1},\bPsi_{i+1})\in  {\mathcal E}^{1,p}_0(B_\rho(0))$ as the solution to 
\be\nonumber
\lf\{
\begin{array}{l}
\Delta A_{i+1}=\tilde{L}(A_i,\bB_i,\bC_i,\bPsi_i) \\[5mm]
\Delta \bB_{i+1}=\tilde{L}(A_i,\bB_i,\bC_i,\bPsi_i) \\ [5mm] 
\Delta \bPsi_{i+1} = \frac{1}{(\la-\bar{\la})} \lf[\lf< \nabla \bB_{i+1} \times \nabla^\perp \bP\rg> + \lf<\nabla A_{i+1}, \nabla^\perp \bP \rg> \rg]  \\[5mm]
\Delta \bC_{i+1}=  \div \lf( \pi_{\bn}(\nabla^2 \bPsi_{i+1})  \res_g \nabla^\perp \bP \times \bn \rg)\quad.
\end{array}
\rg.
\ee
Analogously as for the a priori estimates above, we estimate that 
$$\|A_{i+1}\|_{W^{1,p}}+\|\bB_{i+1}\|_{W^{1,p}}\leq \gamma \varepsilon_0 \|(A_{i},\bB_{i},\bC_{i},\bPsi_{i})\|_{{\mathcal E}^{1,p}_0}$$
which yelds
$$
\|\bPsi_{i+1}\|_{W^{2,p}}+\|\bC_{i+1}\|_{W^{1,p}}\leq \gamma \left(\|A_{i+1}\|_{W^{1,p}}+\|\bB_{i+1}\|_{W^{1,p}} \right) \leq  \gamma' \varepsilon_0  \, \|(A_{i},\bB_{i},\bC_{i},\bPsi_{i})\|_{{\mathcal E}^{1,p}_0} \quad.
$$
Combining the last two estimates we obtain that there exists $\gamma>0$, and for every $\varepsilon_0$ there exists $r>0$ such that for every $\rho\in (0,r)$ it holds
\be
\|(A_{i},\bB_{i},\bC_{i},\bPsi_{i})\|_{{\mathcal E}^{1,p}_0(B_\rho(0))}\leq (\gamma \varepsilon_0)^i \; \|(A_{0},\bB_{0},\bC_{0},\bPsi_{0})\|_{{\mathcal E}^{1,p}_0(B_\rho(0)) } \quad.
\ee
Choosing $\varepsilon_0\leq \frac{\gamma}{2}$ and $r>0$ accordingly, it is straightforward to check that quantities
$$A^0:=\sum_{i=0}^\infty A_i,\quad \bB^0:=\sum_{i=0}^\infty \bB_i,\quad \bC^0:=\sum_{i=0}^\infty \bC_i,\quad \bPsi^0:=\sum_{i=0}^\infty \bPsi_i $$
are well defined with  $(A^0,\bB^0,\bC^0,\bPsi^0)\in  {\mathcal E}^{1,p}_0(B_\rho(0))$, and that $L(A^0,\bB^0,\bC^0,\bPsi^0)=(f_A,\vec{f}_{\bB}, \vec{f}_{\bC},\vec{f}_{\bPsi})$ as desired. 

\end{proof}

As should be clear from the lemmas above, for proving the regularity it is convenient to work with functions with zero boundary value. As we will soon see, to this aim it is enough we add to the system \eqref{eq:defL} some ${\mathcal E}^{-1,4}$-terms coming from the data $(S_F,\bR_F,\bD,\bP)$. From now on, let $r>0$ be given by Lemma \ref{lem:isomorph} and $\rho\in (0,r)$.

Recall that, by Proposition \ref{prop:SysRSFR322}, $(S_F,\bR_F,\bD,\bP)\in {\mathcal E}^{1,2}_{loc}(\D^2)$ therefore, by Fubini's Theorem, for a.e. $\rho\in (0,r)$ we have that  $S_F \in W^{1,2}(\p B_\rho(0),\R)$, $\bR_F\in W^{1,2}(\p B_\rho(0),\R)$, $\bD\in W^{1,2}(B_\rho(0),\R^3)$ and $\bP\in W^{2,2}\cap W^{1,\infty}(\p B_\rho(0),\R^3)$. So, let 
\be \label{eq:defS0}
(S_0,\bR_0,\bD_0, \bP_0)\in {\mathcal E}^{1,4}(B_\rho(0))
\ee  be the extensions to $B_\rho(0)$ of the above quantities defined on $\p B_\rho(0)$ and define $(S^0,\bR^0,\bD^0,\bP^0)\in {\mathcal E}^{1,2}_0(B_\rho(0))$ as
\be\label{eq:defphi0}
(S^0,\bR^0,\bD^0,\bP^0):=(S_F-S_0,\bR_F-\bR_0,\bD-\bD_0,\bP-\bP_0) \quad.
\ee
Let us stress that $(S^0,\bR^0,\bD^0,\bP^0)$ have zero boundary value on $\p B_\rho(0)$. Recalling that  thanks to Proposition \ref{prop:SysRSFR322} it holds $L(S_F,\bR_F,\bD,\bP)=0$, we infer that 
\be\label{eq:LS0}
L(S^0,\bR^0,\bD^0,\bP^0)= (f_S, \vec{f}_{\bR},\vec{f}_{\bD},\vec{f}_{\bP}) \quad \text{on } B_\rho(0) 
\ee
for some $(f_S, \vec{f}_{\bR},\vec{f}_{\bD},\vec{f}_{\bP})\in  {\mathcal E}^{-1,4}(B_\rho(0))$ easy to compute from the definition of $L$ as in \eqref{eq:defL}.
\\Now, thanks to the Isomophism Lemma \ref{lem:isomorph} applied with $p=4$, there exists $(A^0,\bB^0,\bC^0,\bPsi^0)\in  {\mathcal E}^{1,4}_0(B_\rho(0))$ solution to  the system
\be\label{eq:LA0}
L(A^0,\bB^0,\bC^0,\bPsi^0)=(f_S, \vec{f}_{\bR},\vec{f}_{\bD},\vec{f}_{\bP}).
\ee  
But, since clearly ${\mathcal E}^{1,4}_0(B_\rho(0))\subset {\mathcal E}^{1,2}_0(B_\rho(0))$, the uniqueness statement of the Isomorphism Lemma \ref{lem:isomorph} applied this time with $p=2$ together with \eqref{eq:LS0} and \eqref{eq:LA0} implies that 
$$(S^0,\bR^0,\bD^0,\bP^0)=(A^0,\bB^0,\bC^0,\bPsi^0) \quad \text{on } B_\rho(0) \; \Rightarrow \; (S^0,\bR^0,\bD^0,\bP^0)\in {\mathcal E}^{1,4}_0(B_\rho(0)).$$
Therefore, recalling \eqref{eq:defS0} and  \eqref{eq:defphi0}, we conclude that
\be
(S_F,\bR_F,\bD,\bP)\in {\mathcal E}^{1,4}(B_\rho(0)) \quad.
\ee 

Now, plugging the information that $\bP\in W^{2,4}(B_\rho(0))$ in the Euler-Lagrange equation of the frame energy $\F$, we obtain that 
$$\Delta \bH=\bF \text{ on } B_\rho(0)$$
for some $\bF\in W^{-1,2}(B_{\rho}(0))$, so $\bH \in W^{1,2}_{loc}(B_\rho(0))$. Recall that $\Delta \la=(\nabla \bbe_1, \nabla^\perp \bbe_2)\in L^2(B_\rho(0))$, so $\la \in W^{2,2}_{loc}(B_\rho(0))$. Also $\Delta \bP=e^{-2\la} \bH \in W^{1,2}_{loc}(B_\rho(0))$, so $\bP \in W^{3,2}_{loc}(B_\rho(0))$ and in particular $\bP \in W^{2,p}_{loc}(B_\rho(0))$ for every $p\in(1,\infty)$. 
\\Now repeating the above argument, we get that $\bP\in W^{3,p}_{loc}(B_\rho(0))$ for every $p\in(1,\infty)$; the same procedure now gives that  $\bP\in W^{k,p}_{loc}(B_\rho(0))$ for every $p\in(1,\infty)$ and every $k\in \N$. Therefore $\bP$ is smooth in a neighboorod of $0$ and we have proved Theorem \ref{thm:regularity}.
\hfill$\Box$

\section{Existence of a smooth minimizer of $\F$ in regular homotopy classes of tori immersed in $\R^3$: proof of Theorem \ref{thm:MinRegHom} } \label{Sec:RegHom}
At first, it must be proved that the notion of regular homotopy class extend to the general setting of  weak immersions. This is the content of the next proposition.

\begin{Prop}\label{prop:sigmaWeak}
The notion of regular homopy class extends to the framework of weak immersions by approximation. More precisely, let $\bP\in {\cal E}(\T^2,\R^3)$ be a weak immersion; then there exists a sequence $\{\bP_k\}_{k \in \N}$ of smooth immersions and there exists a fixed regular homotopy class $\sigma$ of immersed tori in $\R^3$ such that

a) $\bP_k \in \sigma$ for every $k\in \N$,

b) $\bP_k \to \bP$  in $W^{2,2}(\T^2)\cap W^{1,\infty *}(\T^2)$.
\\
Moreover given any sequence $\{\bP_k\}_{k \in \N}$ satisfying the condition $b)$, then for $k$ large enough also condition $a)$ holds. We therefore define $\sigma$ to be the regular homotopy class of $\bP$.
\end{Prop}

\begin{proof}
The existence of an approximating sequence of smooth  immersions  $\{\bP_k\}_{k \in \N}$ satisfying  condition $b)$ can be achieved by a standard convolution argument recalling the two properties defining a weak immersion (see the Introduction for the definition of weak immersion; more precisely condition 1. ensure that $\bP_k$ is a smooth immersion and the uniform $W^{1,\infty}$ bound on $\bP_k$, and condition 2. implies that $\bP\in W^{2,2}(\T^2, \R^3)$ so that $\bP_k$ converges strongly in $W^{2,2}$-norm). By the strong $W^{2,2}$-convergence it follows that the quantity $\int |d \bn_k|^2 dvol_{g_{\bP_k}}$ does not concentrate. The proof that, for $k_1,k_2$ large enough, two smooth immersions $\bP_{k_1},\bP_{k_2}$ are in the same regular homotopy class is then analogous  to the proof of the no loss of homotopic complexity  in the points of non concentration of the frame energy, namely Case I below.  The independence of the regular homotopy class on the approximating sequence easily follows from the arguments above just by merging  two approximating sequences.
\end{proof}

From now till the end of this section we  fix a regular homotopy class $\sigma$ of immersions of $\T^2$ into $\R^3$ and we consider $\{\bP_k\}_{k \in \N}\subset {\cal E}(\T^2,\R^3)$ a sequence of weak immersions, and $\{\bbe_k\}_{k \in \N}\subset W^{1,2} (\T^2,\bP^*(V_2(\R^3)))$ a sequence of moving frames on $\bP_k(\T^2)$ 
 such that $(\bP_k, \bbe_k)$ is a minimizing sequence of  the frame energy $\F$ among all weak immersions  belonging to  the class $\sigma$,  and all moving frames on them.

Since $(\bP_k,\bbe_k)$ is a minimizing sequence, we can assume that the frame $\bbe_k$ minimizes the tangential frame energy $\F_T$ defined in \eqref{def:FT}, i.e. we can assume that $\bbe_k$ is a Coulomb frame. Using the Chern moving frame technique in order to  construct  conformal coordinates from a Coulomb frame (for more details see \cite{RiCours}) , we get that the weak immersions $\bP_k$ induce a \emph{smooth} conformal structure on $\T^2$; moreover, up to composition with a bilipschitz diffeomorphism of $\T^2$, the weak immersion $\bP_k$ is conformal with respect to this smooth conformal structure.  At this point, analogously to the proof of Lemma \ref{Lm:RedctLB}, we  can assume that the conformal structure is contained in the moduli space $M$ defined in \eqref{eq:defM} and that the moving frame is the coordinate one, i.e. $\bbe_k=\nabla \bP_k \, \frac{\sqrt{2}}{|\nabla \bP_k |}$.
\\

Let us observe that the conformal factors $\la_k:=\log(|\p_{x_i} \bP_k|)$ satisfy the uniform bound
\be\label{eq:Boundlak0}
\sup_{k\in \N} \|\la_k-\bar{\la}_k\|_{L^\infty(\Sigma_k)}+\|\nabla \la_k\|_{L^2(\Sigma_k)} <\infty\quad,
\ee
where $\Sigma_k$ is the flat torus corresponding to the conformal structure of $\bP_k$ and $\bar{\la}_k$ is the mean value of $\la_k$ on $\Sigma_k$. In order to obtain   \eqref{eq:Boundlak}  recall that the conformal factors satisfy  $\Delta \la_k=-<\nabla^{\perp} \bbe_k^1, \nabla \bbe_k^2>$; therefore, by Wente estimates  \cite{We},  we infer
$$\|\la_k-\bar{\la}_k\|_{L^\infty(\Sigma_k)}+\|\nabla \la_k\|_{L^2(\Sigma_k)} \leq C_0 \|\nabla \bbe_k^1 \|_{L^2(\Sigma_k)} \|\nabla \bbe_k^2 \|_{L^2(\Sigma_k)} \leq C_1 \F(\bP_k,\bbe_k)\leq C_2\quad.$$
Notice that if we rescale $\bP_k$ by a factor $e^{-\bar{\la}_k}$ we get that the conformal factors of the rescaled immersions are uniformly bounded in  $L^\infty(\Sigma_k)$; since the frame energy $\F$ is invariant under rescaling, we can replace the minimizig sequence with the rescaled one, so that we can assume
\be\label{eq:Boundlak}
\sup_{k\in \N} \|\la_k\|_{L^\infty(\Sigma_k)}+\|\nabla \la_k\|_{L^2(\Sigma_k)} <\infty\quad.
\ee
 Recalling \eqref{eq:F=FT+bbI}, we have that the the second fundamental forms of the $\bP_k$'s  are uniformly bounded in $L^2(\Sigma_k)$ and therefore, thanks to \eqref{eq:Boundlak}, we infer
\be\label{eq:PkW22}
\sup_{k\in \N}\|\bP_k\|_{W^{2,2}(\T^2,\R^3)}< \infty\quad.
\ee

Now we claim that the conformal structures are contained in a compact subset of the moduli space.  
\\To this aim, observe that the proof of Proposition \ref{Prop:LowerBound1}  can be repeated for weak immersions (just notice that for a.e. $x$ the curve  $\gamma_x$ is $W^{2,2}$, so we can apply Fenchel Theorem and then integrate in $x$; same argument for $y$. All the other computations in the proof makes sense a.e. so the integrated inequality holds as well). Since by definition of $M$ we have $\theta \in \lf[\frac{\pi}{3}, \frac{2\pi} {3}\rg],$ if  $\tau_2 \to \infty$ then the right hand side of \eqref{eq:LB1} diverges to $+\infty$, which implies the claim.   Therefore, up to subsequences in $k$, the conformal structures $\Sigma_k$ converge smoothly in the moduli space to a limit $\Sigma=\Sigma_{\infty}$.  
\\

Combining the convergence of the conformal structures  and the estimates \eqref{eq:Boundlak}-\eqref{eq:PkW22} , we infer that there exists a weak conformal immersion $\bP=\bP_\infty\in {\cal E}(\T^2, \R^3)$, with conformal factor $\la=\la_{\infty}=\log{|\p_{x_i}\bP|}$ and coordinate moving frame $\bbe=\bbe_{\infty}=e^{-\la}\nabla \bP$ such that, up to subsequences,
\be\label{eq:Convk}
\bP_k  \rightharpoonup \bP \, \text{  weakly}-W^{2,2}(\Sigma), \quad \bbe_k \rightharpoonup \bbe  \, \text{  weakly}-W^{1,2}(\Sigma), \quad \la_k \rightharpoonup \la  \, \text{  weakly}-W^{1,2}(\Sigma)\cap L^\infty(\Sigma)^*.
\ee
Let us stress that the limit $\bP$ is not branched thanks to the uniform estimates on the conformal factors \eqref{eq:Boundlak}, which of course pass to the limit under the above converge. 
\\Using the conformal invariance of the Dirichelet integral, the lower semicontinuity of the $L^2$-norm under weak convergence, and the smooth convergence of the conformal structures, it follows that
\begin{eqnarray}
\F(\bP,\bbe)&=&\frac{1}{4}\sum_{i=1}^2 \int_{\Sigma} |d \bbe_i |_{g_{\bP}}^2 \, dvol_{g_{\bP}}= \frac{1}{4}\sum_{i=1}^2 \int_{\Sigma} |\nabla \bbe_i |^2 \, dx \leq \liminf_{k} \frac{1}{4} \sum_{i=1}^2 \int_{\Sigma_k} |\nabla \bbe^i_k |^2 \, dx\nonumber \\
 &=& \liminf_{k} \frac{1}{4}\sum_{i=1}^2 \int_{\Sigma_k} |d \bbe^i_k| _{g_{\bP_k}}^2 \, dvol_{g_{\bP_k}}=\liminf_{k} \F(\bP_k,\bbe_k)\quad. \label{eq:LSC}
\end{eqnarray}
Since $(\bP_k,\bbe_k)$ is by construction a minimizing sequence, thanks to  \eqref{eq:LSC}, in order to finish the proof of Theorem \ref{thm:MinRegHom} we just have to show that  the weak immersion $\bP$ is an element of the regular homotopy class $\sigma$. The regularity of $\bP$ will then follow from Theorem \ref{thm:regularity} and from the criticality (actually we have even minimality) of $\bP$ for the frame energy $\F$.
\\

In order to prove that there is no loss of homotopic complexity in the limit, we are going to show that we can cover $\Sigma$ with a finite number of balls and that on every ball there is no loss of homotopic complexity, with good control of the boundary; in oder to do so,  we start with detecting the points of energy concentration  for the frame energy.
\\Let $\e_0>0$ small to be chosen later;  for every $x \in \Sigma$ and  $k\in \N$ we define
\be\label{eq:defrhok}
\rho_{k,x}:=\inf\lf\{ \rho>0: \, \int_{B_{2\rho}(x)} |\nabla \bbe_k|^2 \, dx \geq \e_0  \rg\},
\ee 
where $B_{2\rho}(x)$ is the ball in $\R^2$ of center $x$ and radius $2\rho$ with respect to the flat metric. 
\\For a given $k\in \N$, the collection $\{B_{\rho_{k,x}} (x)\}_{x \in \Sigma}$ forms a Besicovitch covering of $\Sigma$ therefore, by the Besicovitch covering theorem, there exists a finite subcovering  
$\{B_{\rho_{k,x_i^k}} (x^k_i)\}_{i\in I_k}$ such that any point in $\Sigma$ is covered by at most $c_\Sigma\in \N$ balls, where $c_\Sigma$ does not depend on $k \in \N$. In fact,  from the uniform bound on the frame energy with respect to $k\in \N$, the cardinality of $I_k$ is  uniformly bounded in $k$ thus, up to subsequences, we can assume that $I$ is independent of $k$ (and finite) and that for all $ i \in I$ 
\be\label{eq:convxr}
x^k_i \to x_i, \quad \rho_{k,x_i^k}\to \rho_i \quad \text{ as } k\to \infty \quad,
\ee
for some $x_i\in \Sigma$ and $\rho_i \geq 0$. Letting
\be\label{eq:defJ}
J :=\{i\in I: \,  \rho_i =0\}\quad \text{and } I_0=I\setminus J,
\ee 
it is clear that  $\{\overline{B_{\rho_i} (x_i )}\}_{i\in I_0}$ covers $\Sigma$; moreover, for the strict convexity of the euclidean balls, the points in $\Sigma$ which are not contained in $\cup_{i\in I_0} B_{\rho_i} (x_i )$ cannot accumulate and therefore are isolated and hence finite:
\be\label{eq:defai}
\{a_1,\ldots,a_N \}:=\Sigma\setminus \cup_{i\in I_0}B_{\rho_i} (x_i )\quad .
\ee
In order to show that $\bP$ is an element of $\sigma$, we are going to show that there is no loss go homotopic complexity in the limit. We are going to  consider separately the regions of $\Sigma$ where there is energy concentration and where there is not. Before starting with the latter, observe that we can assume that $\bP_k$ and $\bP$ are smooth immersions, indeed, almost by definition (see Proposition \ref{prop:sigmaWeak}), one can approximate a weak immersion via a smooth immersion without changing the regular homotopy class. \\

\emph{Case I: no loss of homotopic complexity in $B_{\rho_i}(x_i)$, $i \in I_0$}.  From \eqref{eq:Convk} and  \eqref{eq:defrhok}-\eqref{eq:convxr}, using  Fubini's Theorem (and a standard selection argument ensuring the independence of $k$, see for instance \cite[Lemma B.1]{SiL}) we have that there exists $\rho\in (\rho_i,2\rho_i)$ such that, up to subsequences in  $k$, it holds
\be\label{eq:EstConv}
\sup_{k}\int_{\p B_\rho(x_i)} |\nabla \bbe_k|^2 \, dl \leq 2\e_0 \quad \text{and} \quad \bP_k  \rightharpoonup \bP \, \text{  weakly}-W^{2,2}({\p B_\rho(x_i)}.
\ee
Recalling that we can write $|\nabla \bbe_k|^2=|\nabla \bn_k|^2+ 2|\nabla \la_{k}|^2$, by Schwartz inequality we infer that
\be\label{eq:nablanpB}
\int_{\p B_\rho(x_i)} |\nabla \bn_k| dl + \int_{\p B_\rho(x_i)} |\nabla \la_k| \, dl\leq \sqrt{C \e_0 \rho}\quad,
\ee 
for some universal $C>0$. In particular, called  $\k_g$ the geodesic curvature, it  follows that
$$\lf| \int_{\bP_k(\p B_\rho(x_i))} |\k_g| \, dl-2\pi \rg|\leq \sqrt{C \e_0 \rho}.$$
The combination of the last two estimates implies that $\bP_k(\p B_\rho(x_i))$ is a graph over a planar simple closed  curve $\vec{\alpha}(\cdot):S^1\to \R^3$ (which, up to a rotation, we can assume lying on the plane $\R^2=\{z=0\}\subset \R^3$) and, thanks to \eqref{eq:EstConv}, the same holds for $\bP(\p B_\rho(x_i))$. Therefore, up to a regular homotopy, we can assume that $\vec{\alpha}$ is parametring a round circle and that both $\bP_k$ and  $\bP$ coincide with $\vec{\alpha}$ up to first order, i.e. 
\be\label{eq:alpha}
 \bP(\cdot)=\bP_k(\cdot) =\vec{\alpha}(\cdot)\;\text{ and }\;  \p_r\bP(\cdot)= \p_r \bP_k(\cdot) =\frac{\p}{\p r} \in \R^2 \;  \text{on } \,  \p B_\rho(x_i)\quad.
\ee
At this point,  Lemma \ref{Lm:DiskHomotopy} in the appendix  concludes the proof  of  Case I.
\\

\emph{Case II: no loss of homotopic complexity in the concentration points $\{a_1,\ldots,a_N\}$.}  We are going to show the claim at a fixed concentration point $a_1$, of course the argument for the other $a_i$'s is  analogous. By definition of concentration point, there exists a sequence of radii $\rho_k\downarrow 0$ such that 
$$\liminf_k \int_{B_{\rho_k}(a_1)} |\nabla \bn_k|^2 dx \geq \e_0 \quad.$$
Moreover, by the finiteness of  the frame  energy, it is easy to construct  a sequence $\{R_k\}_{k\in \N}$ with the following properties: 
\be\label{eq:Rk}
R_k\downarrow 0,\quad R_k>\rho_k, \quad \lim_{k\to \infty} \frac{\rho_k}{R_k}= 0 \quad \text{and} \quad  \lim_{k\to \infty} \int_{B_{R_k}(a_1)\setminus B_{\rho_k}(a_1) }  |\nabla \bbe_k|^2 dx =0\quad. 
\ee
Now let us rescale the sequence $\bP_k$ by defining
\be\label{eq:defphihat}
\hat{\bP}_k(x):=\frac{1}{R_k} \bP_k(a_1+R_k(x-a_1))\quad.
\ee 
Observe that, by the invariance of the frame energy under scaling,   \eqref{eq:Rk} implies that for every $\delta\in(0,1/4)$ we have
\be\label{eq:frameto0}
0=\lim_{k\to \infty} \int_{B_{1}(0)\setminus B_{\delta}(0) }  |\nabla \hat{\bbe}_k|^2 dx =\lim_{k\to \infty} \int_{B_{1}(0)\setminus B_{\delta}(0) }  |\nabla \hat{n}_k|^2 + 2|\nabla \hat{\la}_k|^2 dx ,
\ee
where of course $\hat{\la}_k=\log|{\p_{x_1}\hat{\bP}_k}|$,  $\hat{\bbe}_k=e^{-\hat{\la}_k} \nabla \hat{\bP}_k$ is the coordinate moving frame associated to $\hat{\bP}_k$ and $\hat{\bn}_k$ is the normal vector.  In order to compare the regular homotopy type, let us also rescale the limit $\bP=\bP_\infty$ given in \eqref{eq:Convk} by the same factors, i.e. we define
\be\label{eq:phiinfk}
\hat{\bP}^k_\infty:=\frac{1}{R_k} \bP_\infty(a_1+R_k(x-a_1))\quad.
\ee 
Since by \eqref{eq:LSC} the frame energy of $\bP_\infty$ is finite, we also have
\be\label{eq:frameinfto0}
0=\lim_{k\to \infty}\int_{B_{R_K}(0)}|\nabla \bbe_\infty|^2 \, dx=\lim_{k\to \infty} \int_{B_{1}(0)}  |\nabla \hat{\bbe}^k_\infty|^2 dx =\lim_{k\to \infty} \int_{B_{1}(0) }  |\nabla \hat{n}^k_\infty|^2 + 2|\nabla \hat{\la}^k_\infty|^2 dx \quad,
\ee
with obvious meaningg of the hatted quantities. 
Let $\hat{\e}_0>0$ small to be fixed later. Combining \eqref{eq:frameto0} and \eqref{eq:frameinfto0}, analogously to Case I (using Fubini's Theorem, a selection argument and Schwartz inequality), we get that there exists $\hat{\rho} \in (1/4,1)$ such that, up to subsequences in $k$, it holds
\be\label{eq:hatprho}
\int_{\p B_{\hat{\rho}}(0) }   |\nabla \hat{n}_k| + |\nabla \hat{\la}_k|+|\nabla \hat{n}^k_\infty| + |\nabla \hat{\la}^k_\infty|\,  dl \leq \hat{\e}_0 \quad. 
\ee
Now,analogously to Case I, we get that  for $\hat{\e}_0$ small enough-or, in other words, for  $k$ large enough- $\bP_k$ and $\bP^k_\infty$ are graphs over  a planar simple closed curve; hence, up to a regular homotopy, we can assume that they coincide up to first order with a planar round circle as in \eqref{eq:alpha}. By  using Lemma \ref{Lm:DiskHomotopy}, we conclude that $\hat{\bP}_k|_{B_{\hat{\rho}}(0)}$ and $\hat{\bP}^\infty_k|_{B_{\hat{\rho}}(0)}$ are regularly homotopic with good control on the boundary homotopy; therefore,  rescaling back by $R_k$, we obtain the same statement for $\bP_k|_{B_{\hat{\rho}R_k}(a_1)}$ and $\bP_\infty|_{B_{\hat{\rho}R_k}(a_1)}$, as desired. 

\begin{Rm}
As a side remark let us observe that, by refining the estimates of Case II and by a cutting and filling procedure-adapted to the frame energy-analogous to the proof of \cite[Lemma 5.2] {MoRi1}, it is possible to prove that actually Case II does not occur. Indeed it is possible to  replace $\bP_k(B_{\rho_k}(a_1)) $ by a flat disk without changing the regular homotopy type and saving $\e_0/2>0$ energy. This would clearly contradict the assumption that $\bP_k$ is a minimizing sequence.  Since this argument is not needed and it is a bit more complicated than Case II discussed above, we decided  to present  this simpler proof.
 \end{Rm}

Summarizing, we proved that $\bP_k$ and $\bP_\infty \in {\cal E}(\T^2, \R^3)$ are elements of the same regular homotopy class $\sigma$. 
 Therefore, by the lower semicontinuity \eqref{eq:LSC}, $\bP_\infty$ is a minimizer of the frame energy $\F$ in his regular homotopy class among weak immersions and $W^{1,2}$-moving frames. In particular $\bP_\infty$ is a critical point of $\F$, and by Theorem \ref{thm:regularity} we conclude that $\bP_\infty$ is smooth. This completes the proof of Theorem \ref{thm:MinRegHom}. 
\hfill$\Box$

\section{Appendices}
\subsection{Appendix A: some classical computations in conformal coordinates}
In this appendix we consider a weak conformal immersion $\vec{\Phi}$ of the disk $D^2$  into $\R^3$. We will denote with  $g$  the pull back metric on $D^2$ induced by the immersion $\vec{\Phi}$. We will use local positive conformal coordinates $x^1, x^2$ on $D^2$  and we will call $(\vec{e}_1,\vec{e}_2)$ the local orthonormal frame such that $\p_{x_1}\vec{\Phi}=e^\la\,\vec{e}_1$ and $\p_{x_2}\vec{\Phi}=e^\la\,\vec{e}_2$; $\lambda:=|\partial_{x^1} \vec{\Phi}|=|\partial_{x^2} \vec{\Phi}|$ is called conformal factor.

\subsubsection{Variation of classical geometric quantities}
The computations in the present subsection are rather classical, we repeat them here mainly to fix the notations. Given a vector field $\vec{w}\in C^\infty_c(\R^3,\R^3)$, consider the one   parameter family of weak  immersions of $D^2$ into $\R^3$ given by $\vec{\Phi}_t:=\vec{\Phi}+t\vec{w}$.

The normal vector $\bn_t$ of  $\vec{\Phi}$ is given by
\be
\label{III.4}
\vec{n}_t=\star_{{\R}^3}\lf(\frac{\p_{x_1}\vec{\Phi}_t\wedge\p_{x_2}\vec{\Phi}_t}{|\p_{x_1}\vec{\Phi}_t\wedge\p_{x_2}\vec{\Phi}_t|}\rg)=\bbe_1(t) \times \bbe_2(t)\quad ,
\ee
and  can be expanded as
\be
\label{III.4a}
\vec{n}_t=\vec{n}+t\ (a_1\,\vec{e}_1+a_2\,\vec{e}_2) +o(t)\quad.
\ee
Since $\vec{n}_t\perp\p_{x_i}\vec{\Phi}_t$ and $\p_{x_i}\vec{\Phi}_t=\p_{x_i}\vec{\Phi}+t\,\p_{x_i}\bw$, we have
\be
\label{III.4c}
(\p_{x_i}\bw, \bn)+ a_i\, e^\la=0 \quad.
\ee 
Combining (\ref{III.4a})  and (\ref{III.4c}) gives then

\be
\label{III.9}
\vec{n}_t=\vec{n}-t  <(d\bw,\bn) \; \bn,d\vec{\Phi}>_g+o(t)\quad,
\ee
which gives
\be
\label{III.10}
\frac{d\vec{n}_t}{dt}(0)=- <(d\bw,\bn) \; \bn,d\vec{\Phi}>_g\quad .
\ee

We have $g_{ij}=(\p_{x_i}\vec{\Phi}_t, \p_{x_j}\vec{\Phi}_t)$, thus
\be
\label{0III.10f}
\frac{d}{dt}g_{ij}(0)=(\p_{x_i}\vec{w}, \p_{x_j}\vec{\Phi})+ (\p_{x_i}\vec{\Phi}, \p_{x_j}\vec{w}).
\ee
Since $\sum_{i}g^{ki}g_{ij}=\delta_{kj}$ and since $g_{ij}(0)=e^{2\la}\, I_2$ where $I_2$ is the $2\times 2$ identity matrix, differentiating  we get
\be
\label{0III.10h}
\frac{d}{dt}g^{kj}(0)=-e^{-4 \la } \frac{d}{dt} g_{kj}= - \,e^{-4\la}\, \lf[ (\p_{x_k}\vec{w}, \p_{x_j}\vec{\Phi})+ (\p_{x_k}\vec{\Phi}, \p_{x_j}\vec{w})\rg] \quad.
\ee
We also have 
\be
\label{III.3a}
\begin{array}{l}
\ds\frac{d}{dt}\lf(dvol_{g_t}\rg)=\frac{d}{dt}\lf(det(g_{ij})\rg)^{1/2}dx_1\wedge dx_2=2^{-1}\lf(det(g_{ij})\rg)^{-1/2}\,e^{2\la}\,\lf[\frac{dg_{11}}{dt}+\frac{dg_{22}}{dt}\rg]\, dx_1\wedge dx_2\\[5mm]
\ds\quad\quad\quad\quad\quad= <d\bP,d\bw>_g \, dvol_{g_0} \quad.
\end{array}
\ee

\subsubsection{Codazzi identity in complex coordinates}
Called $z=x_1+ix_2$ the complex coordinate associated to $(x_1,x_2)$ and  $\p_z=2^{-1}(\p_{x_1}-i\,\p_{x_2})$, 
the Weingarten form of the weak conformal immersion $\vec{\Phi}\in {\cal E}(D^2,\R^3)$ is defined as
\begin{equation}\label{eq:defh0}
\vec{h}^0:=2\, \pi_{\vec{n}}(\p^2_{z^2}\vec{\Phi})\ dz\otimes dz\quad; 
\end{equation}
the scalar Weingarten form (in case of codimension one immersions) is $h^0:=2 (\vec{n}, \p^2_{z^2}\vec{\Phi})\ dz\otimes dz$.
\\Now let us state and prove the classical Codazzi indentity using the complex coordinates. 

\begin{Lm}[Codazzi identity]
\label{Codazzi}
Let $h^0$ be the scalar Weingarten form defined above  and denote $g_{{\C}}:=e^{2\la}\ d\ov{z}\otimes dz$. Then it holds
\be
\label{xsIII.21}
\ov{\p}h^0=g_{{\C}}\otimes\p H\quad.
\ee
\hfill $\Box$
\end{Lm}

\begin{proof}
First of all observe that 
\begin{equation}\label{eq:h0H0}
\vec{h}^0:=2\, \pi_{\vec{n}}(\p^2_{z^2}\vec{\Phi})\ dz\otimes dz =e^{2\la}\ \vec{H}^0\ dz\otimes dz \quad,
\end{equation}
where
\be\label{eq:H0}
\vec{H}^0=2\,\p_z\lf(e^{-2\la}\p_z\vec{\Phi}\rg)=2^{-1} e^{-2\la}\ \pi_{\vec{n}}(\p^2_{x^2_1}\vec{\Phi}-\p^2_{x^2_2}\vec{\Phi})-
i\ e^{-2\la}\ \pi_{\vec{n}}(\p^2_{x_1x_2}\vec{\Phi})
=[H^0_{\Re}+i\,H^0_{\Im}]\ \vec{n}\quad.
\ee
This gives
\be
\label{xsIII.17}
\lf\{
\begin{array}{l}
\ds H^0_{\Re}=-\frac{e^{-2\la}}{2}\,\lf[(\p_{x_1}\vec{n},\p_{x_1}\vec{\Phi})-(\p_{x_2}\vec{n},\p_{x_2}\vec{\Phi})\rg]=e^{-2\la}{\mathbb I}^0_{11}=-e^{-2\la}{\mathbb I}^0_{22}\\[5mm]
\ds H^0_{\Im}=e^{-2\la}\,(\p_{x_1}\vec{n},\p_{x_2}\vec{\Phi})=e^{-2\la}\,(\p_{x_2}\vec{n},\p_{x_1}\vec{\Phi})=-e^{-2\la}{\mathbb I}^0_{12}\quad.
\end{array}
\rg.
\ee
We have
\[
\begin{array}{l}
\p_{\ov{z}}\vec{H}^0=2\,\p_{\ov{z}}\p_z\lf(e^{-2\la}\,\p_z\vec{\Phi}\rg)=2\,\p_{{z}}\p_{\ov{z}}\lf(e^{-2\la}\,\p_z\vec{\Phi}\rg)\\[5mm]
\quad=-4\,\p_z\lf(e^{-2\la}\,\p_{\ov{z}}\la\,\p_z\vec{\Phi}\rg)+2\,\p_z\lf(e^{-2\la}\,\p^2_{z\ov{z}}\vec{\Phi}\rg)=-4\,\p_z\lf(e^{-2\la}\,\p_{\ov{z}}\la\,\p_z\vec{\Phi}\rg)+2^{-1}\,\p_z\lf(e^{-2\la}\,\Delta\vec{\Phi}\rg) \quad.
\end{array}
\]
Hence we have obtained the following identity
\be  \nonumber
\p_{\ov{z}}\vec{H}^0=-4\,\p_z\lf(e^{-2\la}\,\p_{\ov{z}}\la\,\p_z\vec{\Phi}\rg)+\p_z\vec{H}-\p_z\vec{\Phi}\quad.
\ee
(Notice that the term $\p_z\vec{\Phi}$ comes from the fact that for $\bP:\T^2\hookrightarrow \Sp^3 \subset \R^4$ we have $\frac 12 e^{-2\la}\Delta \bP= \bH-\bP$; on the other hand if $\bP:\T^2\hookrightarrow \R^3$ one has $\frac 12 e^{-2\la}\Delta \bP= \bH$ so the term $\p_z\vec{\Phi}$ is not present).  Taking the scalar product with $\vec{n}$ gives then
\be \nonumber
\p_{\ov{z}}H^0=-2\,\p_{\ov{z}}\la\,H^0+\p_zH
\ee
from which we deduce
\be 
\label{xsIII.20}
\p_{\ov{z}}\lf(e^{2\la}\,H^0\rg)= e^{2\la}\ \p_zH\quad,
\ee
which is the thesis. Notice that the identity can also be rewritten locally as 
\be
\label{xsIII.20.0}
\lf\{
\begin{array}{l}
\ds\p_{x_1}{\mathbb I}^0_{11}+\p_{x_2}{\mathbb I}^0_{12}=e^{2\la}\ \p_{x_1}H\\[5mm]
\ds\p_{x_2}{\mathbb I}^0_{11}-\p_{x_1}{\mathbb I}^0_{12}=-e^{2\la}\ \p_{x_2}H \quad.
\end{array}
\rg.
\ee
\end{proof}

\subsection{Appendix B: a lemma of functional analysis}
 
In this appendix, for the reader's convenience, we recall the following (known but maybe non completely standard)  lemma of functional analysis which  plays a key  role in the proof of the regularity.
\begin{Lm}\label{lem:Appendix(p,p)}
Let $b \in W^{1,2}(D^2)$ and $p \in (1,\infty)$ be given; then for every $a \in W^{1,p}(D^2)$ there exists a unique solution $\varphi\in W^{1,p}(D^2)$  of the equation 
\be\label{eq:syst(p,p)}
\lf\{ 
\begin{array}{l}
\Delta \varphi = \partial_x a \, \p_y b- \p_y a  \, \p_x b \quad \text{on } D^2 \\[5mm]
\varphi =0 \quad \text{on  } \p D^2,
\end{array}
\rg.
\ee
moreover tha linear map $L^p(D^2)\ni \nabla a \mapsto \nabla \varphi \in L^p(D^2)$ is continuous.
\end{Lm}  

\begin{proof}
As first step let us assume $p>2$. By H\"older inequality, we have that $\nabla a \, \nabla^\perp b \in L^{\frac{2p}{p+2}}(D^2)$ so by classical elliptic theory there exists a unique solution $\varphi \in W^{2, \frac{2p}{p+2}}(D^2)$ and $\|\varphi\|_ {W^{2, \frac{2p}{p+2}}} \leq C \|\nabla a\|_{L^p} \, \|\nabla b\|_{L^2}$. We conclude  by Sobolev embedding.

Now let us assume $p<2$. Let $\bar{a}$ be the mean value of $a$ on $D^2$. Observe that the system \eqref{eq:syst(p,p)} is equivalent to the following one
\be\label{eq:syst(p,p)'}
\lf\{ 
\begin{array}{l}
\Delta \varphi = \partial_x ( (a-\bar{a}) \, \p_y b)- \p_y ((a-\bar{a})  \, \p_x b) \quad \text{on } D^2 \\[5mm]
\varphi =0 \quad \text{on  } \p D^2,
\end{array}
\rg.
\ee
By the Poincar\'e-Sobolev inequality we know that $\|a-\bar{a}\|_{L^{\frac{2p}{2-p}}} \leq C \|a\|_{L^p}$, so by H\"older inequality we get that 
$$\|\nabla((a-\bar{a}) \nabla^\perp b)\|_{L^p}\leq C \|a\|_{L^p} \|b\|_{L^2} $$
and we conclude with classical elliptic theory.

Finally, the continuity of the map  $L^2(D^2)\ni \nabla a \mapsto \nabla \varphi \in L^2(D^2)$ follows by the classical Riesz-Torin interpolation Theorem.
\end{proof}

\subsection{Appendic C: a lemma of differential topology}
In this subsection we recall the following  well known Lemma of classical  differential topology; we  wish to thank Brian White and Yasha Eliashberg for an helpful discussion about this point.
\begin{Lm}\label{Lm:DiskHomotopy}
Let us identify $\R^2$ with $\{z=0\}\subset \R^3$. Let $D^2$ be the unit disk in  $\R^2$ and $\bP$  be a smooth immersion of $D^2$ into $\R^3$ such that $\bP|_{\p D^2}=Id_{\p D^2}$ and $\p_{r} \bP|_{\p D^2}= \frac{\p}{\p r}\in \R^2$. 

Then there exists a regular homotopy from $\bP$ to the identity map of $D^2$, $Id_{D^2}$, relative to $\p D^2$; in other words there exists  $\vec{H}(\cdot,\cdot)\in C^1(\D^2\times [0,1], \R^3)$ such that for every $t\in [0,1]$ the map $\bH_t:=\bH(\cdot,t):D^2\to \R^3$ is an immersion satisfying the boundary conditions $\bH_t|_{\p D^2}=Id_{\p D^2}, (\p_r \bH)|_{\p D^2} =\frac{\p}{\p r}$,  and moreover   $\vec{H}_0=\bP$, $\vec{H}_1=Id_{D^2}$.  
\end{Lm}

\begin{proof}
Let us sketch the main idea of the classical proof. The work of Smale \cite{Smale58}-\cite{Smale59} reduces the classification of the regular homotopy classes of immersions of the disk $D^k$ into $\R^n$ (fixed near the boundary) to the computation of the homotopy group $\pi_k(V_k(\R^n))$, where $V_{k}(\R^n)$ is the Stiefel  manifold of $k$-frames in $\R^n$. For $k=2, n=3$ we have $\pi_2(V_{2}(\R^3))=\pi_2(SO(3))=0$.

This reduction is performed by the so called h-principle. The main references are the classical paper of Hirsch  \cite{Hirsch59}  as well as Gromov's  book \cite{GroPDR} and the recent  book of  Eliashberg-Mishachev \cite{ElMih}.
\end{proof}

\end{document}